\documentclass[11pt]{amsart}

\usepackage[english]{babel}
\usepackage[english=american]{csquotes}

\usepackage{amsmath, amssymb, amsthm}
\newtheorem{theorem}{Theorem}[section]
\newtheorem{proposition}[theorem]{Proposition}
\newtheorem{lemma}[theorem]{Lemma}
\newtheorem{corollary}[theorem]{Corollary}
\newtheorem{fact}[theorem]{Fact}
\theoremstyle{definition}
\newtheorem{definition}[theorem]{Definition}
\theoremstyle{remark}

\newtheorem{case}{Case}

\usepackage{enumitem}
\setlist[enumerate]{label=(\roman*)}
\setlist[itemize]{leftmargin=1cm}
\usepackage{csquotes}

\usepackage[dvipsnames]{xcolor}
\usepackage{lipsum}
\usepackage[colorinlistoftodos]{todonotes}

\usepackage{graphicx}
\usepackage{tikz, pgfplots}
\pgfplotsset{compat=1.18}
\usetikzlibrary{arrows.meta}

\usepackage{hyperref}

\DeclareMathOperator{\Mann}{Mann}
\DeclareMathOperator{\Carmichael}{Carmichael}
\DeclareMathOperator{\Inequ}{Inequ}
\DeclareMathOperator{\BInequ}{BInequ}
\DeclareMathOperator{\Th}{Th}

\DeclareMathOperator{\Sub}{Sub}

\DeclareMathOperator{\tp}{tp}
\DeclareMathOperator{\fr}{frac}
\DeclareMathOperator{\lcm}{lcm}

\renewcommand{\div}{^\mathrm{div}}
\newcommand{\id}{\mathrm{id}}

\author{Philipp Hieronymi}
\address{Mathematical Institute\\ University of Bonn\\
Endenicher Allee 60\\ 53115 Bonn\\ Germany}
\email{hieronymi@math.uni-bonn.de}

\author{Michael Reitmeir}
\address{Mathematical Institute\\ University of Bonn\\
Endenicher Allee 60\\ 53115 Bonn\\ Germany}
\email{reitmeir@math.uni-bonn.de}

\author{Xiaoduo Wang}
\address{Department of Mathematics\\
University of Manchester\\
Oxford Road\\
Manchester\\
M13 9PL\\ UK}
\email{xiaoduo.wang@manchester.ac.uk}

\title[Axiomatizations of Presburger Arithmetic With Powers]{Axiomatizations of Presburger Arithmetic With Predicates For Powers}

\begin{document}

\maketitle

\begin{abstract}
  We give a complete first-order axiomatization of the structure \((\mathbb{Z},+,(\ell^{\mathbb{N}})_{\ell\in L})\), where \(L \subseteq \mathbb{Z}_{\ge 2}\) is a set of pairwise multiplicatively independent integers and \(\ell^{\mathbb{N}} = \{\ell^n : n\in \mathbb{N}\}\). Using recent work of Karimov et al.\ \cite{Karimov}, we obtain that this axiomatization is computable for \(|L|=2\), which proves that \((\mathbb{Z},+,k^{\mathbb{N}}, \ell^{\mathbb{N}})\) is decidable for \linebreak[2]\(k\), \(\ell\in \mathbb{Z}_{\ge 2}\). Furthermore, we give an axiomatization of the universal theory of \((\mathbb{Z},+,<,(\ell^{\mathbb{N}})_{\ell\in L})\).
\end{abstract}

\section{Introduction}

In 1929, Presburger \cite{Presburger} gave a complete and computable axiomatization of the first-order theory of the structure \((\mathbb{Z},+,<)\), which proves its decidability. Indeed, this theory, nowadays referred to as \emph{Presburger arithmetic}, is completely described by the axioms of a discretely ordered abelian group as well as \emph{congruence axioms} which express that out of \(n\) consecutive integers, exactly one is divisible by \(n\). We now focus on expansions of Presburger arithmetic. This has been an active topic of research to this day, and we refer the reader to B\`es \cite{Bes-Survey} for an extensive survey of this area. Here, we are interested in expansions by predicates for sets of powers \(\ell^{\mathbb{N}} = \{\ell^n : n \in \mathbb{N}\}\) for \(\ell \in \mathbb{Z}_{\ge 2}\). It follows from Büchi's theorem \cite{Buchi} on the decidability of the weak monadic second-order theory of $(\mathbb{N},+1)$ that the first-order theory of \((\mathbb{Z},+,<,\ell^{\mathbb{N}})\) and even the first-order theory of the more expressive structure \((\mathbb{Z},+,<,V_{\ell})\) are decidable for every \(\ell \in \mathbb{Z}_{\ge 2}\), where \(V_{\ell} \colon \mathbb{Z} \to \mathbb{Z}\) maps \(0\) to \(1\) and \(n \ne 0\) to the largest power of \(\ell\) that divides \(n\). See Bruy\`ere et al. \cite{BHMV, BHMV-2} for a proof.

Progress on expansions of Presburger arithmetic by two or more sets of powers has been slower. Let \(k\) and \(\ell \in \mathbb{Z}_{\ge 2}\) be multiplicatively independent. Here, Büchi's theorem is not applicable. Indeed, the structure \((\mathbb{Z},+,<,V_k, V_{\ell})\) defines multiplication by Villemaire \cite{Villemaire} and thus its first-order theory is undecidable. This result is strengthened in B\`es \cite{Bes}, where it is shown that already \((\mathbb{Z},+,<,V_k, \ell^{\mathbb{N}})\) defines multiplication. The undecidability of the first-order theory of \((\mathbb{Z},+,<, k^{\mathbb{N}}, \ell^{\mathbb{N}})\) has only been established recently in Hieronymi and Schulz \cite{HieronymiSchulz}, partly because this structure does not define multiplication by Schulz \cite{Schulz}. Here, we show decidability of the first-order theory when we drop the order relation.

Another way to analyze expansions of Presburger arithmetic is in terms of model-theoretic tameness. The structure \((\mathbb{Z},+,<,\ell^{\mathbb{N}})\) is NIP for every \(\ell \in \mathbb{Z}_{\ge 2}\), which is proven in Lambotte and Point \cite{LambottePoint} and Hawthorne \cite{Hawthorne} using two different methods. Similarly to decidability, adding more than one predicate for powers breaks this tameness, as \((\mathbb{Z},+,<,k^{\mathbb{N}},\ell^{\mathbb{N}})\) has IP by \cite{HieronymiSchulz}. However, the situation is different if one omits the order relation: The first proof of the stability of \((\mathbb{Z},+,\ell^{\mathbb{N}})\) appears as a special case of a more general theorem in Moosa and Scanlon \cite{MoosaScanlon}, though this structure is not explicitly mentioned there. This is improved in Palac\'in and Sklinos \cite{PalacinSklinos}, showing that \((\mathbb{Z},+,\ell^{\mathbb{N}})\) is superstable of \(U\)-rank \(\omega\). By Conant \cite{Conant}, the structure \((\mathbb{Z},+,(\ell^{\mathbb{N}})_{\ell \in L})\) remains superstable of \(U\)-rank \(\omega\) for any finite subset \(L \subseteq \mathbb{Z}_{\ge 2}\), which implies that the expansion \((\mathbb{Z},+,(\ell^{\mathbb{N}})_{\ell \ge 2})\) by powers of all positive integers is stable.

Our contribution to this picture is as follows: Let \(L \subseteq \mathbb{Z}_{\ge 2}\) be a set of pairwise multiplicatively independent integers. We give an axiomatization of \((\mathbb{Z},+,-,0,1,(\ell^{\mathbb{N}})_{\ell \in L})\) and, using the recent decidability results of Karimov, Luca, Nieuwveld, Ouaknine and Worrell \cite{Karimov}, we show that this set of axioms is computable for \(|L| = 2\). Thus, the first-order theory of \((\mathbb{Z},+,k^{\mathbb{N}},\ell^{\mathbb{N}})\) is decidable. Moreover, we prove a relative quantifier elimination for \((\mathbb{Z},+,(\ell^{\mathbb{N}})_{\ell \in L})\), which yields superstability of \((\mathbb{Z},+,(\ell^{\mathbb{N}})_{\ell\ge 2})\).

If we add the order relation back in, undecidability of the first-order theory implies that we cannot hope for a computable axiomatization of \((\mathbb{Z},+,<,(\ell^{\mathbb{N}})_{\ell \in L})\) for \(|L| \ge 2\). However, for \(|L| = 2\), the structure is existentially decidable by \cite{Karimov}, and thus also universally decidable. Here, we give an axiomatization of the universal theory of \((\mathbb{Z},+,-,0,1,<,(\ell^{\mathbb{N}})_{\ell \in L})\). In the case \(|L| = 2\), this axiomatization is computable.

Let us briefly sketch our axioms before defining them in detail in Section \ref{ch:axioms}. Perhaps the most intuitive property satisfied by powers of an integer \(\ell\) is 
\[1 \in \ell^{\mathbb{N}} \wedge \forall x \left( x \in \ell^{\mathbb{N}} \leftrightarrow \ell \cdot x \in \ell^{\mathbb{N}}\right).\]
In addition to this and the axioms of Presburger, we will need three more complex axiom schemata. In the spirit of quantifier elimination for Presburger arithmetic, they describe the equations, congruences, and inequalities satisfied by powers. Let \(\ell\), \(\ell_1\), \dots, \(\ell_m \in L\).
\begin{itemize}
  \item Let \(a_1\), \dots, \(a_m\), \(b \in \mathbb{Z}\). It is a known number-theoretic fact that a linear equation like \(a_1x_1 + \dots + a_mx_m = b\) has only finitely many non-degenerate solutions \((x_1, \dots, x_m) \in \ell_1^{\mathbb{N}} \times \dots \times \ell_m^{\mathbb{N}}\), that is, solutions such that \(\sum_{i \in I} a_ix_i \ne 0\) for all non-empty \(I \subseteq \{1, \dots, m\}\). The \emph{Mann axioms} hard-code these finitely many solutions.
  \item For \(m \in \mathbb{N}\), every element of \(\ell^{\mathbb{N}}\) except for \(\ell^0\), \(\ell^1\), \dots, \(\ell^{m-1}\) is divisible by \(\ell^m\). On the other hand, if we look at the congruences of \(\ell^{\mathbb{N}}\) modulo some \(n \in \mathbb{N}\) which is coprime to \(\ell\), then the sequence \((\ell^m)_{m \in \mathbb{N}}\) is periodic and does not take certain values (e.g.,\ the value \(0\)) modulo \(n\). The \emph{Carmichael axioms} capture these phenomena.
  \item In \((\mathbb{Z},+,<,(\ell^{\mathbb{N}})_{\ell \in L})\), the \emph{inequality axioms} hard-code all the systems of strict homogeneous linear inequalities which do not have a solution in \(\ell_1^{\mathbb{N}} \times \dots \times \ell_m^{\mathbb{N}}\). If \(|L| = 2\), then, following the algorithm for deciding the solvability of such inequalities in \cite{Karimov}, we show that a small subset of these axioms suffices. These \emph{basic inequality axioms} state that for every \(x \in \ell^{\mathbb{N}}\), we have \(x > 0\), and there is no \(y \in \ell^{\mathbb{N}}\) such that \(x < y < \ell \cdot x\).
\end{itemize}
In Section \ref{ch:prelims}, we will fix the languages and notation used for the remainder of this article. Afterwards, in Section \ref{ch:axioms}, we define the three aforementioned axiom schemata precisely, discuss how they are used, and show that they are computable for \(|L|=2\). In Section \ref{ch:small}, we prove that the Mann axioms imply that sets of powers are \emph{small} in a certain technical sense. This is then applied in Section \ref{ch:orderless}, where we define the theory \(T(L)\) and prove that it axiomatizes the first-order theory of \((\mathbb{Z},+,(\ell^{\mathbb{N}})_{\ell \in L})\). We do this by giving a back-and-forth system, which we also use to prove our relative quantifier elimination. Finally, in Section \ref{ch:ordered}, we define the theory \(T_\forall(L)\) and prove that it axiomatizes the universal theory of \((\mathbb{Z},+,<,(\ell^\mathbb{N})_{\ell \in L})\).

\subsection*{Acknowledgements} The authors were supported by the Deutsche \linebreak Forschungsgemeinschaft (DFG, German Research Foundation) under Germany's Excellence Strategy - GZ 2047/1, Projekt-ID 390685813. The authors thank the Hausdorff Research Institute
for Mathematics for its hospitality during the trimester program ``Definability, decidability, and computability''. In the early stages of this research, P.H. was partially supported by NSF grant DMS-1654725. Preliminary versions of the results presented here appear in the Master's theses of M.R. and X.W. The authors thank Toghrul Karimov for explaining the algorithms from \cite{Karimov} to them.
\section{Notation and Preliminaries}
\label{ch:prelims}
Let \(\mathbb{N} = \{0,1,2, \dots\}\). Throughout, \enquote{definable} means definable with parameters. All theories from now on are first-order theories.

For \(L \subseteq \mathbb{Z}_{\ge 2}\), we define \(\mathcal{L} := \{+,-,0,1\} \cup \{U_{\ell} : \ell \in L\}\) to be the language of additive groups with the constant symbol \(1\) and a unary relation symbol \(U_{\ell}\) for each \(\ell \in L\), whose standard interpretation will be \(\ell^{\mathbb{N}}\). Furthermore, let \(\mathcal{L}_{<} := \mathcal{L} \cup \{<\}\).

We introduce the following abbreviations for convenience and clarity, without adding new symbols to our language. Every \(\mathcal{L}\)-structure \(\mathcal{M}\) that we consider will be a torsion-free abelian group, meaning we can and will identify the subgroup generated by \(1^{\mathcal{M}}\) with \(\mathbb{Z}\). For \(n\in \mathbb{Z}\), let \(n\cdot x\) abbreviate \(x + \dots + x\) with \(n\) summands for \(n\ge 0\), and \(-x - \dots -x\) with \(-n\) summands for \(n < 0\). We simply write \(n\) instead of \(n\cdot 1\). For \(n\in \mathbb{Z}\), let \(D_n\) be the unary relation symbol interpreted as \(D_n(x)\) if and only if \(\exists y (n\cdot y = x)\). To avoid clutter, we will often write an \(\mathcal{L}\)-structure (resp.\ an \(\mathcal{L}_{<}\)-structure) \(\mathcal{M}\) as \((\mathcal{M}, (A_{\ell})_{\ell\in L})\) instead of \((M,+,-,0,1,(A_{\ell})_{\ell\in L})\) (resp.\ \((M,+,-,0,1,<,(A_{\ell})_{\ell\in L})\)), by which we implicitly mean \(A_\ell = U_\ell^\mathcal{M}\). If \((\mathcal{M}, (A_{\ell})_{\ell\in L})\) is an \(\mathcal{L}\)-structure, we define \(A := \bigcup_{\ell\in L} A_{\ell} \). If \(x = (x_1, \dots, x_n)\), then we define \(x > 0\) to mean \(x_i > 0\) for all \(i\in \{1, \dots, n\}\).

\begin{definition}
  Let \(G\) be a torsion-free abelian group and let \(H \le G\) be a subgroup. We say \(H\) is \emph{divisibly closed} (also called \emph{pure}) in \(G\) if for all \(n\in \mathbb{N}\) and \(g\in H\), the existence of \(x\in G\) with \(n\cdot x = g\) implies the existence of \(y\in H\) with \(n\cdot y = g\). We define \(H\div\) to be the smallest subgroup of \(G\) that is divisibly closed in \(G\) and contains \(H\).
\end{definition}

\begin{definition}
  Let \(\mathcal{M} = (M,+,-,0,1)\) be a group and let \(g\in M\). We call the partial type
  \begin{align*}
    d_{\mathcal{M}}(g) := \{D_{p^{m}}(x-u) :&\ p\text{ prime}, m>0, 0\le u < p^{m}, \mathcal{M} \models D_{p^m}(g-u)\}
  \end{align*}
  the \emph{congruence type of \(g\) in \(\mathcal{M}\)}. If \((g_i)_{i\in I}\) is a family of elements of \(M\), then we define its congruence type as
  \begin{align*}
    d_{\mathcal{M}}((g_{i})_{i\in I}) := \{D_{p^{m}}(x_i-u) :&\ i\in I, p\text{ prime}, m>0, 0\le u < p^{m},\\ &\ \mathcal{M} \models D_{p^m}(g_i-u)\}.
  \end{align*}
\end{definition}
Note that it will be important that \(d_{\mathcal{M}}(g)\) only contains \emph{positive} formulas of the form \(D_{p^m}(x-k)\), not negative formulas of the form \(\neg D_{p^m}(x-k)\). This makes no difference if \(\mathcal{M}\) satisfies the congruence axioms of Presburger arithmetic. In this case, if \(g\), \(h\in M\) are such that \(h\) satisfies \(d_{\mathcal{M}}(g)\) (or in other words, \({d_{\mathcal{M}}(g) \subseteq d_{\mathcal{M}}(h)}\)), then \(d_{\mathcal{M}}(g) = d_{\mathcal{M}}(h)\). However, if \(\mathcal{M}\) is only a subgroup of a model of Presburger arithmetic (which is the case we will consider in Section \ref{ch:ordered}), then we might have \(g\), \(h\in M\) such that \(d_{\mathcal{M}}(g)\) is a proper subset of \(d_{\mathcal{M}}(h)\).

Throughout, we usually assume that the elements of \(L \subseteq \mathbb{Z}_{\ge 2}\) are pairwise multiplicatively independent. This assumption is more for convenience than for a lack of generality, as demonstrated by the following lemma.

\begin{lemma} \label{thm:dependence}
  Let \(k\), \(\ell \in \mathbb{Z}_{\ge 2}\) be multiplicatively dependent. Then \(\ell^{\mathbb{N}}\) is definable in \((\mathbb{Z},+,k^{\mathbb{N}})\).
\end{lemma}
\begin{proof}
  By assumption, we have \(k^m = \ell^n\) for some \(m\), \(n \in \mathbb{Z}_{\ge 1}\). Notice that \(1\) is definable in \((\mathbb{Z},+,k^{\mathbb{N}})\) by the formula \(U_k(x) \wedge \neg D_k(x)\). Using this, we can define \((k^m)^{\mathbb{N}}\) in \((\mathbb{Z},+,k^{\mathbb{N}})\) by 
  \[U_k(x) \wedge D_{k^m-1}(x-1).\]
  Finally, \(\ell^{\mathbb{N}}\) is definable in \((\mathbb{Z},+,(\ell^n)^{\mathbb{N}})\) by
  \[\exists y \left( U_{\ell^n}(y) \wedge \bigvee_{t=0}^{n-1} \ell^t\cdot x = y \right).\vspace{-1em}\]
\end{proof}

\section{Axioms}
\label{ch:axioms}

In this section, we lay out the axiom schemata that we use to axiomatize the first-order theory of \((\mathbb{Z},+,-,0,1,(\ell^{\mathbb{N}})_{\ell\in L})\) in Section \ref{ch:orderless} and the universal theory of \((\mathbb{Z},+,-,0,1,<,(\ell^{\mathbb{N}})_{\ell\in L})\) in Section \ref{ch:ordered}. 

\subsection{Mann Axioms}
\begin{definition}[Mann property] \label{def:mann-property}
  Let \(K\) be a field and let \(G\) be a subset of \(K^{\times }\). For \(a_1\), \dots, \(a_n\), \(b\in \mathbb{Z}\), consider the equation
  \begin{equation}
  \label{eq:mann}
    a_1x_1 + \dots + a_nx_n = b.
  \end{equation}
  \begin{sloppypar}
  A \emph{solution in \(G\)} of this equation is a tuple \((g_1, \dots, g_n)\in G^n\) with \(a_1g_1 + \dots + a_ng_n = b\). A solution \((g_1, \dots, g_n)\) is called \emph{non-degenerate} if \(\sum_{j\in J} a_jg_j \ne 0\) for every non-empty \(J \subseteq \{1, \dots, n\}\). We say that \(G\) has the \emph{Mann property} if for every \(n \ge 2\) and \(a_1, \dots, a_n, b\in \mathbb{Z}\), the equation \eqref{eq:mann} has only finitely many non-degenerate solutions in \(G\).
  \end{sloppypar}
\end{definition}

The Mann property was first introduced in van den Dries and Günaydin \cite{vdDGunaydin}. We have slightly rephrased it to be more suitable for our work over \(\mathbb{Z}\). In particular, we also define it for subsets \(G \subseteq K^{\times}\), whereas in \cite{vdDGunaydin} only \emph{subgroups} \(G \le K^{\times }\) were considered. The Mann property is named after Henry B.\ Mann, who proved that the group of roots of unity in \(\mathbb{C}\) has this property \cite{Mann}.

It is known that every finite-rank subgroup of the multiplicative group of a field of characteristic \(0\) has the Mann property; see van der Poorten and Schlickewei \cite{PoortenSchlickewei}, Evertse \cite{Evertse}, and Laurent \cite{Laurent}. (Indeed, by Evertse, Schlickewei and Schmidt \cite{EvertseSchlickeweiSchmidt}, there is even an explicit upper bound on the number of non-degenerate solutions depending only on the number of variables and the rank of the group.) This is what we make use of in our axiomatizations.

Let \(L\) be a subset of \(\mathbb{Z}_{\ge 2}\). The idea behind the Mann axioms is to hard-code the finitely many solutions of \eqref{eq:mann} in \(\bigcup_{\ell\in L} \ell^{\mathbb{Z}}\). However, since we cannot express division in our language, we need to add another variable to the right-hand side and focus on equations of the form
\begin{equation}
  a_1x_1 + \dots + a_nx_n = by \label{eq:mann-hom}
\end{equation}
instead, where we are interested in solutions in \(\bigcup_{\ell\in L} \ell^{\mathbb{N}}\). These equations are homogeneous and may thus have infinitely many solutions: For example, if \(s_1\), \dots, \(s_n\), \(t\in \ell^{\mathbb{N}}\) satisfy \(a_1s_1 + \dots + a_ns_n = bt\), then so do \(\ell^ks_1\), \dots, \(\ell^ks_n\), \(\ell^kt\) for any \(k\in \mathbb{N}\). The Mann property, however, guarantees that these potentially infinitely many integral solutions \mbox{\(s_1\), \dots, \(s_n\), \(t\)} correspond to only finitely many rational solutions \(\frac{s_1}{t}\), \dots, \(\frac{s_n}{t}\). We can even do a little better if the elements of \(L\) are pairwise multiplicatively independent:

\begin{lemma} \label{thm:still-fin-many-sol}
  Let \(a_1, \dots, a_n, b\in \mathbb{Z}\) with \(b \ne 0\), and let \(L \subseteq \mathbb{Z}_{\ge 2}\) be a set of pairwise multiplicatively independent numbers. If \(\ell_1\), \dots, \(\ell_n\), \(\ell_{n+1}\in L\) are not all equal, then there are only finitely many non-degenerate tuples \((s_1, \dots, s_n, t)\) such that
  \(a_1s_1 + \dots + a_ns_n = bt\)
  and \(s_1\in \ell_1^{\mathbb{N}}\), \dots, \(s_n\in \ell_n^{\mathbb{N}}\), \(t \in \ell_{n+1}^{\mathbb{N}}\).
\end{lemma}
\begin{proof}
  As above, we refer to
  \begin{itemize}
    \item \((s_1, \dots, s_n, t)\) with \(a_1s_1 + \dots + a_ns_n = bt\) and \(s_1\in \ell_1^{\mathbb{N}}\), \dots, \(s_n\in \ell_n^{\mathbb{N}}\), \(t \in \ell_{n+1}^{\mathbb{N}}\) as \emph{integral solutions}, and
    \item to \((r_1, \dots, r_n)\) with \(a_1r_1 + \dots + a_nr_n = b\) and \(r_i \in \{\ell_i^{j}/\ell_{n+1}^k : j,k\in \mathbb{N}\}\) for \(i\in \{1, \dots, n\}\) as \emph{rational solutions}.
  \end{itemize}
  The Mann property tells us that there are only finitely many non-degenerate rational solutions. It thus suffices to show that two different integral solutions lead to different rational solutions.

  Let \((s_1, \dots, s_n, t)\) and \((s_1', \dots, s_{n}', t')\) be integral solutions with
  \[ \left(\frac{s_1}{t}, \dots, \frac{s_n}{t} \right) = \left(\frac{s_1'}{t'} , \dots, \frac{s_n'}{t'} \right).\]
  By assumption, there is \(j\in \{1, \dots, n\}\) such that \(\ell_j \ne \ell_{n+1}\), so \(\ell_j\) and \(\ell_{n+1}\) are multiplicatively independent. Therefore, \(s_j/t = s'_j/t'\) implies that \(s_j = s'_j\) and \(t = t'\). The latter implies \(s_i = s_i'\) for all \(i\in \{1, \dots, n\}\).
\end{proof}

For the Mann axioms, we thus distinguish two cases:
\begin{itemize}
  \item If \(\ell_1\), \dots, \(\ell_n\), \(\ell_{n+1}\in L\) are not all the same, then we hard-code the finitely many non-degenerate integral solutions.
  \item If \(\ell_1 = \dots = \ell_n = \ell_{n+1}\), then we hard-code the finitely many non-degenerate rational solutions.
\end{itemize}
To do the latter in a canonical way, we say that \((s_1, \dots, s_n, t) \in (\ell^{\mathbb{N}})^{n+1}\) is \emph{primitive} if at least one of the fractions \(s_i/t\) for \(i\in \{1, \dots, n\}\) is reduced, so in other words, if \(s_i = 1\) for some \(i\in \{1, \dots, n\}\) or \(t = 1\).

\begin{definition}[Mann axioms]\label{def:mann-axioms}
  Let \(L \subseteq \mathbb{Z}_{\ge 2}\) be a set of pairwise multiplicatively independent numbers, and let \(n\ge 1\), \(a_1\), \dots, \(a_n\), \(b\in \mathbb{Z}\setminus \{0\}\), and \(\ell_1\), \dots, \(\ell_n\), \(\ell_{n+1}\in L\). We abbreviate \(\overline{a} = (a_1, \dots, a_n)\), \(\overline{\ell} = (\ell_1, \dots, \ell_{n+1})\), and \(x = (x_1, \dots, x_{n+1})\), the latter being a tuple of variables. Let \(\theta_{\overline{a},b,\overline{\ell}}(x)\) be the formula
  \[\bigwedge_{i=1}^{n+1} U_{\ell_i}(x_i) \wedge \sum_{i=1}^n a_ix_i = bx_{n+1} \wedge \bigwedge_J \sum_{j\in J} a_jx_j\ne 0,\]
  where the conjunction is taken over all non-empty \(J \subseteq \{1, \dots, n\}\).
  \begin{itemize}
    \item Suppose \(\ell_1\), \dots, \(\ell_n\), \(\ell_{n+1}\) are not all equal, and let
      \[(s_{11}, \dots, s_{1n},t_1), \dots, (s_{m1}, \dots, s_{mn}, t_m)\]
    be the finitely many non-degenerate solutions of \(a_1x_1 + \dots + a_nx_n = by\) with
    \((s_{i1}, \dots, s_{in},t_i)\in \ell_1^{\mathbb{N}} \times \dots \times \ell_{n+1}^{\mathbb{N}}\)
    for all \(i\in \{1, \dots, m\}\). Then we define \(\Mann(\overline{a},b, \overline{\ell})\) as the sentence
      \[\forall x \left[ \theta_{\overline{a},b,\overline{\ell}}(x) \rightarrow \bigvee_{i=1}^m \Biggl( \Biggl(\bigwedge_{j=1}^n x_j = s_{ij}\Biggr) \wedge x_{n+1} = t_i  \Biggr)\right].\]
    \item Suppose \(\ell := \ell_1 = \dots = \ell_{n+1}\), and let
      \[(s_{11}, \dots, s_{1n},t_1), \dots, (s_{m1}, \dots, s_{mn}, t_m)\]
      be the finitely many non-degenerate primitive solutions of the equation \(a_1x_1 + \dots + a_nx_n = by\) with
      \((s_{i1}, \dots, s_{in},t_i)\in (\ell^{\mathbb{N}})^{n+1}\) for all \(i\in \{1, \dots, m\}\). In this case, we define \(\Mann(\overline{a},b, \overline{\ell})\) to be the sentence
      \[\forall x \left[ \theta_{\overline{a},b,\overline{\ell}}(x) \rightarrow \bigvee_{i=1}^m \bigwedge_{j=1}^n t_{i} x_j = s_{ij}x_{n+1} \right].\]
  \end{itemize}
  We call \(\Mann(\overline{a},b,\overline{\ell})\) the \emph{Mann axiom} corresponding to \(\overline{a}\), \(b\), \(\overline{\ell}\), and define \(\Mann(L)\) to be the axiom schema consisting of \(\Mann(\overline{a},b,\overline{\ell})\) for all \(n\ge 1\), \(\overline{a}\in (\mathbb{Z} \setminus \{0\})^n\), \(b\in \mathbb{Z} \setminus \{0\}\), \(\overline{\ell}\in L^{n+1}\).
\end{definition}

Before moving on, we want to quickly sketch one way the Mann axioms will be used.
Suppose \((\mathcal{M},(A_{\ell})_{\ell\in L}) \models \Mann(L)\). Take \(a_1\), \dots, \(a_n\), \(b\in \mathbb{Z}\) and \(g_1\), \dots, \(g_n \in A \setminus \mathbb{Z}\). Then \(a_1g_1 + \dots + a_n g_n = b\) can only hold if \(b=0\). Indeed, if \(b \ne 0\), then we can remove summands from the left-hand side until we obtain a non-degenerate solution. Now if there is no \(\ell\in L\) with \(g_i\in A_{\ell}\) for all \(i\in\{1, \dots, n\}\), the Mann axioms tell us that \((g_1, \dots, g_n, 1)\) is one of our finitely many hard-coded solutions. But these are all integers, which contradicts our assumption that \(g_1\), \dots, \(g_n \notin \mathbb{Z}\). So we have that \linebreak[2]\(g_1\), \dots, \(g_n \in A_{\ell}\) for some \(\ell\in L\), in which case the Mann axioms tell us that for each \(i\in\{1, \dots, n\}\), there are \(s\), \(t\in \ell^{\mathbb{N}}\) such that \(t\cdot g_i = s\cdot 1\). This yields \(g_i \in \mathbb{Z}\), which is again a contradiction.\footnote{Note that this argument implicitly used the additional assumptions that \(\mathcal{M}\) is a torsion-free abelian group, that \(1\in A_{\ell}\) for all \(\ell\in L\), and that any \(s\in \mathbb{Z}\) that is not divisible by \(t\in \mathbb{Z}_{\ge 2}\) in \(\mathbb{Z}\) is also not divisible by \(t\) in \(\mathcal{M}\) (i.e., \(\mathcal{M} \cap \mathbb{Q} = \mathbb{Z}\)). These assumptions will always be satisfied in the structures we consider.}

Using \cite{Karimov}, we now show that the set \(\Mann(L)\) can be effectively computed for \(L = \{k,\ell\}\), where \(k\), \(\ell\in \mathbb{Z}_{\ge 2}\) are multiplicatively independent.

\begin{definition}[{\cite[Definition 7]{Karimov}}] \label{def:structure-of-solutions}
  A set \(X \subseteq \mathbb{N}^n\) belongs to the class \(\mathfrak{U}\) if it can be written in the form
  \[X = \bigcup_{i\in I} \bigcap_{j \in J_i} X_j\]
  where \(I\) and \(J_i\) for every \(i\in I\) are finite, and each \(X_j\) is either of the form
  \begin{equation}
  \label{eq:coupled}
    X_j = \left\{(e_1, \dots, e_n)\in \mathbb{N}^n : e_{\mu(j)} = e_{\sigma(j)} + c_j\right\}
  \end{equation}
  or of the form
  \begin{equation}
  \label{eq:fixed}
    X_j = \left\{ (e_1, \dots, e_n)\in \mathbb{N}^n : e_{\xi(j)}=b_{j}\right\}
  \end{equation}
  where \(\xi(j), \mu(j), \sigma(j) \in \{1, \dots, n\}\) and \(b_j, c_j\in \mathbb{N}\).
\end{definition}

\begin{lemma}[{\cite[Theorem~8]{Karimov}}]
  \label{thm:diophantine-decidability}
  Let \(k\), \(\ell \in \mathbb{\mathbb{Z}}_{\ge 2}\) be multiplicatively independent integers, and let \(z_1\), \dots, \(z_n\in \{k, \ell\}\), \(C \in \mathbb{Z}^{s \times n}\), and \(d \in \mathbb{Z}^s\). Let \(\mathcal{S}\) be the set of solutions \((e_1, \dots, e_n)\in \mathbb{N}^n\) of the equation \(Cz = d\), where \(z = (z_1^{e_1}, \dots, z_n^{e_n})\).
  Then \(\mathcal{S}\in \mathfrak{U}\), and a representation \(\mathcal{S} = \bigcup_{i\in I} \bigcap_{j \in J_i} X_j\) as in Definition \ref{def:structure-of-solutions} can be effectively computed, with the additional property that \(z_{\mu(j)} = z_{\sigma(j)}\) for every \(X_j\) of the form \eqref{eq:coupled}.
\end{lemma}

To show that the Mann axioms \(\Mann(\{k, \ell\})\) are computable, the only remaining task is to effectively filter out the degenerate solutions from \(\mathcal{S}\). As we will see in the proof of the following lemma, the representation from Definition \ref{def:structure-of-solutions} is already quite convenient for that: every \(\bigcap_{j\in J_i} X_j\) contains either only degenerate solutions or precisely one (primitive) non-degenerate solution.

\begin{lemma} \label{thm:nondegenerate-decidability}
  Let \(k\), \(\ell \in \mathbb{Z}_{\ge 2}\) be multiplicatively independent, let \(z_1\), \dots, \(z_n\), \(z_{n+1}\in \{k, \ell\}\), and let \(a_1,\) \dots, \(a_n\), \(b \in \mathbb{Z} \setminus \{0\}\).
  \begin{enumerate}
    \item If \(z_1\), \dots, \(z_{n+1}\) are not all equal, then the finitely many non\hskip0pt-\hskip0pt degenerate solutions of \(a_1x_1 + \dots + a_nx_n = by\) in \(z_1^{\mathbb{N}} \times \dots \times  z_{n+1}^{\mathbb{N}}\) can be effectively computed. \label{item:integral}
    \item If \(z := z_1 = \dots = z_{n+1}\), then the finitely many non-degenerate primitive solutions of \(a_1x_1 + \dots + a_nx_n = by\) in \((z^{\mathbb{N}})^{n+1}\) can be effectively computed.\label{item:rational}
  \end{enumerate}
\end{lemma}
\begin{proof}
  Let \(\mathcal{S}\) be the set of solutions \((e_1, \dots, e_{n+1}) \in \mathbb{N}^{n+1}\) of
  \[a_1 z_1^{e_1} + \dots + a_nz_n^{e_n} = b z^{e_{n+1}}_{n+1}.\]
  By Lemma \ref{thm:diophantine-decidability}, we can effectively compute \(\mathcal{S} = \bigcup_{i\in I} \bigcap_{j\in J_i} X_j \in \mathfrak{U}\) as in Definition \ref{def:structure-of-solutions}. Since \(I\) is finite, it suffices to show that for each \(i\in I\), the non-degenerate solutions coming from \(\bigcap_{j\in J_i} X_j\) can be effectively computed. For simplicity, we omit the index \(i\) and write \(X := \bigcap_{j\in J} X_j\). We may assume \(X \ne \emptyset \). For ease of notation, we set \(a_{n+1} := -b\).

  Rewrite the representation from Definition \ref{def:structure-of-solutions} as follows: partition the set \(\{1, \dots, n+1\}\) into disjoint subsets \(I_1, \dots, I_d, K\) and pick \(c_1\), \dots, \(c_{n+1} \in \mathbb{N}\) such that the tuples \((e_1, \dots, e_{n+1}) \in X\) are exactly the ones that satisfy
  \begin{align*}
    e_i + c_i &= e_j + c_j && \text{ for all } m \in \{1, \dots, d\}, i,j \in I_m, \text{ and }\\
    e_i &= c_i && \text{ for all } i \in K.
  \end{align*}
  Note that this new representation of \(X\) can be effectively computed from \((X_j)_{j\in J}\).

  Suppose there is an \(m \in \{1, \dots,d \}\) such that \(\emptyset \ne I_m \ne \{1, \dots, n+1\}\). Define \(e_i' := e_i + 1\) for \(i \in I_m\) and \(e_i' := e_i\) otherwise. Since \[e_i'+c_i = e_i+1+c_i = e_j+1+c_j = e_j'+c_j\] for all \(i\), \(j \in I_m\), we also have \((e_1', \dots, e_{n+1}') \in X\). This implies that
  \[\sum_{i\in I_m} a_i z_i^{e_i} = - \sum_{i \notin I_m} a_i z_i^{e_i} = - \sum_{i \notin I_m} a_i z_i^{e_i'} = \sum_{i \in I_m} a_i z_i^{e_i'} = \sum_{i \in I_m} a_i z_i^{e_i + 1}.\]
  From the additional property given by Lemma~\ref{thm:diophantine-decidability}, we know that there is \(z\in \{k, \ell\}\) such that \(z = z_i \) for all \(i \in I_m\). Thus
  \[0 = \sum_{i\in I_m} a_i z^{e_i+1} - \sum_{i\in I_m} a_i z^{e_i} = (z-1) \sum_{i\in I_m} a_i z^{e_i}.\]
  Since \(z \ne 1\), we obtain that in this case, every solution in \(X\) is degenerate.

  The remaining cases are that \(I_m = \{1, \dots, n+1\}\) for some \(m\), or that \({K = \{1, \dots, n+1\}}\). In the latter case, \(X = \{(c_1, \dots, c_{n+1})\}\), and we can easily check whether this single solution is degenerate. So suppose that \(I_m = \{1, \dots, n+1\}\). The additional property in Lemma~\ref{thm:diophantine-decidability} again tells us that there is \(z\in \{k, \ell\}\) such that \(z = z_i \) for all \(i \in \{1, \dots, n+1\}\), meaning we are in the situation of \ref{item:rational}. Now for every \((e_1, \dots, e_{n+1})\), \((e_1', \dots, e_{n+1}') \in X\), we have \[e_i - e_{n+1} = c_{n+1}-c_i = e_i' - e_{n+1}'\] for all \(i\in \{1, \dots, n\}\). This means all \((e_1, \dots, e_{n+1}) \in X\) yield the same rational solution
  \[\left( z^{c_{n+1}-c_1}, \dots , z^{c_{n+1}-c_n}\right).\]
  The corresponding primitive integral solution is
  \[ \left( z^{c_{n+1}-c_1 + c}, \dots, z^{c_{n+1}-c_n +c}, z^c\right), \]
  where \(c := -\min(c_{n+1}-c_1, \dots, c_{n+1}-c_n, 0)\).
  Again, we can easily check whether this single solution is degenerate.
\end{proof}

\subsection{Carmichael Axioms}

In \cite{Carmichael}, Carmichael defined \(\lambda \colon \mathbb{Z}_{>0} \to \mathbb{Z}_{>0}\), nowadays called the \emph{Carmichael function}, as follows: For a prime \(p\), define
\[\lambda(p^v) := \begin{cases}
  2^{v-2} & \text{ if } p=2, v\ge 3\\
  p^{v-1}(p-1) & \text{ otherwise},
\end{cases}.\]
Now for general \(n\in \mathbb{Z}_{>0}\), define \[\lambda(n) := \lcm(\lambda(p_1^{v_1}) \cdots \lambda(p_m^{v_m})),\] where \(n = p_1^{v_1} \cdots p_m^{v_m}\) is the prime factorization of \(n\). The essential property of \(\lambda(n)\) is that it is the smallest positive integer such that \(\ell^{\lambda(n)} \equiv 1 \pmod{n}\) holds for every \(\ell\in \mathbb{Z}\) coprime to \(n\). This means that for \(\ell\) and \(n\) coprime, it suffices to compute \(\ell^1\), \(\ell^2\), \dots, \(\ell^{\lambda(n)} \pmod{n}\) to know all the possible remainders of powers of \(\ell\) modulo \(n\).

\begin{definition}[Carmichael axioms] \label{def:carmichael-axioms}
  For \(L \subseteq \mathbb{Z}_{\ge 2}\), let \(\Carmichael(L)\) be the axiom schema containing for every \(\ell\in L\) the following axioms:
  \begin{itemize}
    \item For \(m\in \mathbb{N}\), the axiom
      \begin{equation}\label{eq:car1}
        \forall x \left[U_{\ell}(x) \rightarrow \left( \Biggl(\bigvee_{j=0}^{m-1} x = \ell^j\Biggr) \vee \left(\bigwedge_{k=1}^{\ell^m-1} \neg D_{\ell^m}(x- k)\right) \right) \right].
      \end{equation}
    \item For \(n\in \mathbb{N}\) coprime to \(\ell\), add the axiom
      \begin{equation} \label{eq:car2}
        \forall x \left[U_\ell(x) \rightarrow \bigwedge_{k\in E} \neg D_{n}(x-k)\right]
      \end{equation}
      where \(E := \left\{ k\in \{1, \dots, n\} : \forall m\in \{1, \dots, \lambda(n)\}: \ell^m \not \equiv k \pmod{n}\right\}.\)
  \end{itemize}
\end{definition}
So in particular, the Carmichael axioms rule out an element of \(U_\ell\) lying in a congruence class that no power of \(\ell\) lies in. Note that, given the congruence axioms of Presburger arithmetic, the formula \eqref{eq:car1} is equivalent to
\[\forall x \left[U_{\ell}(x) \rightarrow \left( \Biggl(\bigvee_{j=0}^{m-1} x = \ell^j\Biggr) \vee D_{\ell^m}(x)\right) \right]\]
which says that \(\ell^0\), \dots, \(\ell^{m-1}\) are the only powers of \(\ell\) not divisible by \(\ell^m\). Furthermore, still modulo Presburger arithmetic, \eqref{eq:car2} is equivalent to
\[\forall x \left[ U_{\ell}(x) \rightarrow \bigvee_{m=1}^{\lambda(n)} D_n(x-\ell^m)\right].\]

The reason for defining the axioms using \(\bigwedge_k \neg D_n(x-k)\) instead of \(D_n(x)\) is that \( \neg D_n(x-k)\) is a universal formula, whereas \(D_n(x)\) contains an existential quantifier. This way, we can also use \(\Carmichael(L)\) as part of our axiomatization of the universal theory of \((\mathbb{Z},+,-,0,1,<,(\ell^{\mathbb{N}})_{\ell\in L})\).

We remark that for computable \(L \subseteq \mathbb{Z}_{\ge 2}\), the set \(\Carmichael(L)\) is computable using the Carmichael function as explained in the first paragraph of this subsection.

The Carmichael axioms will be used in the following way:

\begin{lemma} \label{thm:carmichael-use}
  Let \(L \subseteq \mathbb{Z}_{\ge 2}\), and let \(\mathcal{M} = (M,+,-,0,1,(A_{\ell})_{\ell\in L})\) be a torsion-free abelian group with \(\mathcal{M} \models \Carmichael(L)\). Let \(\ell\in L\), \(a \in A_{\ell} \setminus \ell^{\mathbb{N}}\), and let \(\Delta \subseteq d_{\mathcal{M}}(a)\) be finite. Then there are infinitely many \(n\in \mathbb{N}\) such that \(\ell^n\) satisfies \(\Delta\). In particular, \(\Delta\) is equivalent to
  \[\{ D_{d_1}(x-\ell^n), \dots, D_{d_m}(x-\ell^n)\}\]
  for some prime powers \(d_1\), \dots, \(d_m\) and some \(n\in \mathbb{N}\).
\end{lemma}
\begin{proof}
  We partition \(\Delta\) into two subsets:
  \begin{align*}
    \Delta_1 &:= \left\{ D_{p^m}(x-k) \in \Delta : p \mid \ell\right\}\\
    \Delta_2 &:= \left\{ D_{p^m}(x-k) \in \Delta : p \nmid \ell\right\}
  \end{align*}
  Since \(a \in A_{\ell} \setminus \ell^{\mathbb{N}}\), the Carmichael axiom \eqref{eq:car1} implies that \(\Delta_1\) contains only formulas of the form \(D_{p^m}(x)\) for a prime \(p \mid \ell\) and \(m\in \mathbb{N}\). Pick \(\mu\in \mathbb{N}\) such that \(p^m \mid \ell^{\mu}\) for all \(D_{p^m}(x) \in \Delta_1\).

  By the Chinese remainder theorem, the conjunction of all formulas in \(\Delta_2\) is equivalent to \(D_d(x-k)\) for some \(d\), \(k\in \mathbb{N}\). By the Carmichael axiom \eqref{eq:car2}, there is some \(s \in \mathbb{N}\) such that \(\ell^s \equiv k \pmod{d}\). Now for every \(t\in \mathbb{N}\), we also have \(\ell^{s + \lambda(d)t} \equiv k \pmod{d}\), which means that \(\ell^{s+\lambda(d)t}\) satisfies \(\Delta_2\). Picking \(t\) large enough such that \(s+\lambda(d)t \ge \mu\), we obtain that \(\ell^{s + \lambda(d)t}\) also satisfies \(\Delta_1\), which finishes the proof.
\end{proof}

\subsection{Inequality Axioms}

Finally, we discuss the axioms for the order relation that we will need for the universal theory of \((\mathbb{Z},+,-,0,1,<,(\ell^{\mathbb{N}})_{\ell\in L})\). Naturally, we will use the axioms for a discretely ordered group. Moreover, we will need axioms that describe which inequalities have solutions\linebreak[2] in \((\ell^{\mathbb{N}})_{\ell\in L}\).

\begin{definition}[Inequality axioms] \label{def:inequality-axioms}
  For \(L \subseteq \mathbb{Z}_{\ge 2}\), let \(\Inequ(L)\) be the axiom schema consisting of the sentences
  \[\forall x \left[ \left(\bigwedge_{i=1}^n U_{\ell_i}(x_i)\right) \rightarrow \neg \left(Cx > 0 \wedge  \bigwedge_{i=1}^n D_d(x_i-k_i)\right)\right]\]
  for all \(n \ge 1\), \(\ell_1\), \dots, \(\ell_n\in L\), \(C \in \mathbb{Z}^{m \times n}\), \(d\in \mathbb{Z}_{\ge 2}\), and \(k_1\), \dots, \(k_n\in \{1, \dots, d\}\) such that there is no \(x = (x_1, \dots, x_n) \in \ell_1^{\mathbb{N}} \times \dots \times \ell_n^{\mathbb{N}}\) with \(Cx > 0\) and \(x_i \equiv k_i \pmod{d}\) for all \(i\in \{1, \dots, n\}\).
\end{definition}

The reason why we need to simultaneously axiomatize inequalities and congruences is that an inequality in powers can imply certain congruence conditions. For example, if \(x\), \(y \in 2^{\mathbb{N}}\) and \(x < y < 4x\), then either \(x \equiv 1\), \(y \equiv 2 \pmod{3}\) or \(x \equiv 2\), \(y \equiv 1 \pmod{3}\).
The use of these axioms will be that if \((\mathcal{M}, (A_{\ell})_{\ell\in L}) \models \Inequ(L)\), then for any \(\ell_1\), \dots, \(\ell_n\in L\), a system of linear homogeneous inequalities and congruences has a solution in \(A_{\ell_1} \times  \dots \times A_{\ell_n}\) if and only if it has one in \(\ell_1^{\mathbb{N}} \times \dots \times \ell_n^{\mathbb{N}}\).

In particular, the inequality axioms imply the following, much simpler properties:
\begin{definition}[Basic inequality axioms]\label{def:basic-ineq-axioms} 
  For \(L \subseteq \mathbb{Z}_{\ge 2}\), let \(\BInequ(L)\) be the axiom schema consisting of the sentences
  \[\forall x \forall y \biggl[ \Bigl(U_{\ell}(x) \wedge U_{\ell}(y)\Bigr) \rightarrow \Bigl(x>0 \wedge \neg (x < y < \ell \cdot x)\Bigr) \biggr] \]
  for each \(\ell \in L\).
\end{definition}

We remark that by \cite{Karimov}, the set \(\Inequ(\{k, \ell\})\) is computable for multiplicatively independent \(k\), \(\ell\in \mathbb{Z}_{\ge 2}\). In fact, we will show in Section \ref{sec:solving-inequalities} that \(\BInequ(\{k, \ell\})\) and \(\Carmichael(\{k, \ell\})\) together already imply \(\Inequ(\{k, \ell\})\) (modulo our other axioms). The proof will closely follow the aforementioned decision algorithm. We conjecture that something similar is true for \(|L|>2\).\footnote{This would likely require the additional assumption that \((\log(\ell)^{-1})_{\ell\in L}\) are \(\mathbb{Q}\)-linearly independent, as this is the necessary condition to apply Kronecker's Approximation Theorem.} 

\section{Smallness}
\label{ch:small}
One ingredient in the proof that our axiomatization is complete will be that sets of powers are \emph{small} in the following sense.
\begin{definition}[Smallness]
  Let \(\mathcal{L}\) be a language and \(\mathcal{M}\) be an \(\mathcal{L}\)-structure. Let \(X \subseteq M^m\) and \(Y \subseteq M\).
  \begin{itemize}
      \item For \(n \in \mathbb{N}\), we write \[f \colon X \overset{n}{\to} Y\] to say that \(f\) is a function from \(X\) to \(\mathcal{P}(Y)\) such that for each \(x\in X\), we have \(|f(x)| \le n\). We call \(f\) a \emph{multivalued function}.
      \item For any \(f \colon X \to \mathcal{P}(Y)\), we define the graph of \(f\) as \[G_{f} := \{(x,y)\in X \times Y : y\in f(x)\}\] and the image of a set \(S \subseteq X\) as
      \[f(S) := \bigcup_{x\in S} f(x).\]
      \item If \(X\) and \(Y\) are definable in \(\mathcal{M}\), then \(f\) is said to be \emph{definable} in \(\mathcal{M}\) if \(G_f\) is definable in \(\mathcal{M}\).
      \item Let \(S \subseteq M\). If \(f(S^{m})= M\) for some \(m\), \(n\in \mathbb{N}\) and some definable multivalued function \(f \colon M^m \overset{n}{\to} M\), we say \(S\) is \emph{large} in \(\mathcal{M}\). Otherwise, we say \(S\) is \mbox{\emph{small} in \(\mathcal{M}\).}
  \end{itemize}
\end{definition}

This notion of smallness was also first introduced in \cite{vdDGunaydin}. They show that if \(K\) is an algebraically closed field and \(G \le K^{\times }\) satisfies the Mann property, then \(G\) is small in \(K\) \cite[Lemma 6.1]{vdDGunaydin}. The main goal of this section is to prove the same statement for models of \(\Th(\mathbb{Z},+)\). First, we use quantifier elimination to find a more concrete condition for smallness.

\begin{lemma} \label{thm:large-qe}
  Let \(\mathcal{M}\) be a model of \(\Th(\mathbb{Z},+,-,0,1, (D_n)_{n\ge 2})\), and let \(S \subseteq M\) be large in \(\mathcal{M}\). Then \(M\) is a finite union of sets of the form
  \[\{y\in M : \exists x_1, \dots, x_m\in S \colon  a_1 x_1 + \dots + a_mx_m = by + c\},\]
  where \(m \ge 1\), \(a_1\), \dots, \(a_m\), \(b\in \mathbb{Z}\) with \(b\ne 0\), and \(c\in M\).
\end{lemma}
\begin{proof}
  Since \(S\) is large, there is a definable multivalued function \(f \colon \mathbb{Z}^m \overset{n}{\to} \mathbb{Z}\) with \(f(S^m) = \mathbb{Z}\). Let \(\varphi(x,y,d)\) be the formula defining the graph of \(f\), where \(d\) is a tuple of parameters in \(M\). This means for all \(s\in M^m\) and \(t\in M\), we have \(t\in f(s)\) if and only if \(\mathcal{M} \models \varphi(s,t,d)\).
  Since \(\mathcal{M}\) has quantifier elimination, we may assume that \(\varphi(x,y,d)\) is of the form \(\bigvee_i \bigwedge_{j} \theta_{i,j}(x,y,d)\), where each \(\theta_{i,j}\) is an atomic formula or the negation of an atomic formula.
  So each \(\theta_{i,j}(x,y,d)\) is equivalent to a formula of one of the following three forms, where \(a_1\), \dots, \(a_m\), \(b\in \mathbb{Z}\) and \(c\in M\):
  \begin{enumerate}
      \item \(\sum_{k=1}^{m} a_{k}\cdot x_k = b\cdot y + c\) \label{it:equation}
      \item \(\sum_{k=1}^{m} a_k\cdot x_k \ne b\cdot y + c\) \label{it:inequality}
      \item \(D_p(\sum_{k=1}^{m} a_kx_k + by + c)\) \label{it:divisibility}
  \end{enumerate}
  Note that we do not need to separately consider the negation of \ref{it:divisibility}, as
  \[\mathcal{M} \models \neg D_p(x+k) \leftrightarrow \bigvee_{\substack{1\le u < p\\ u\ne k}}D_{p}(x+u).\]
  Define \(f_{i,j}(s) := \{t\in M : \mathcal{M} \models \theta_{i,j}(s,t)\}\) and \(f_i(s) = \bigcap_j f_{i,j}(s)\) for \(s\in S^m\). Thus, we have \(\bigcup_{i} f_{i}(s) = f(s)\). We may assume that for each \(i\), there is some \(s\in S^m\) such that \(f_i(s)\ne \emptyset\), because otherwise this index \(i\) would be irrelevant.

  For a fixed \(i\), let \(J_1\), \(J_2\), \(J_3\) be the sets of indices \(j\) such that \(\theta_{i,j}\) is of the form \ref{it:equation}, \ref{it:inequality}, \ref{it:divisibility}, respectively. Towards a contradiction, suppose \(J_1 = \emptyset\). Pick some \(s\in S^m\) such that \(f_{i}(s)\ne \emptyset\). Then for \(j \in J_2\), the set \(f_{i,j}(s)\) is cofinite, so \(\bigcap_{j\in J_2} f_{i,j}(s)\) is cofinite as well. Furthermore, the set \(\bigcap_{j\in J_3} f_{i,j}(s)\) consists of the solutions of some system of linear congruences and is thus empty or infinite. So \(f_{i}(s) = \bigcap_{j\in J_2\cup J_{3}} f_{i,j}(s)\) is empty or infinite. Since \(f_{i}(s) \subseteq f(s)\) and \(|f(s)| \le n\), the set \(f_i(s)\) must be finite and thus empty, contradicting our assumption. Furthermore, suppose that for all \(j\in J_1\), the \(b\) in the formula \(\theta_{i,j}\) is equal to zero. Then \(f_{i,j}(s)\) is empty or infinite for all \(j\) and all \(s\in S^m\), which means we get a similar contradiction.

  So we have shown that for every \(i\), there is some \(j(i)\) such that \(\theta_{i,j(i)}(x,y)\) is equivalent to \(\sum_{k=1}^{m}a_{i,k} x_{k} = b_{i}y + c_{i}\), where \(b_i \ne 0\). Therefore,
  \begin{align*}
    M = f(S^m) &= \bigcup_{x\in S^m} \bigcup_i \,\bigcap_j \,\{y\in M : \mathcal{M} \models \theta_{i,j}(x,y,c_{i,j})\}\\
    & \subseteq \bigcup_{x\in S^m} \bigcup_i \,\{y\in M : \mathcal{M} \models \theta_{i,j(i)}(x,y,c_{i,j(i)})\}\\
    &= \bigcup_i \left\{y\in M : \exists x_1, \dots, x_m\in S \colon \sum_{k=1}^m a_{i,k}x_k = b_i y + c_i\right\} \subseteq M.
  \end{align*}
  \qedhere\vspace{-.6em}
\end{proof}

Our next step is to show that, over \((\mathbb{Z},+)\), smallness is equivalent to the following notion from \cite{Conant}.

\begin{definition}[{\cite[Definition 2.4]{Conant}}]
  Let \(S \subseteq \mathbb{Z}\). Define \[{\textstyle\sum_n }(S) := \{s_1+\dots+s_k : k\le n, s_1,\dots,s_k\in S\}\] for \(n\ge 1\). The set \(S\) is called \emph{sufficiently sparse} if \(\sum_n(\pm S)\) does not contain any non-trivial subgroup of \(\mathbb{Z}\) for all \(n\ge 1\), where \(\pm S := \{x, -x : x\in S\}\).
\end{definition}

\begin{lemma} \label{thm:one-more-sparse}
  Let \(S \subseteq \mathbb{Z}\) be sufficiently sparse and \(z\in \mathbb{Z}\). Then \(S \cup \{z\}\) is sufficiently sparse.
\end{lemma}
\begin{proof}
  Suppose \(S \cup \{z\}\) was not sufficiently sparse, so \(b \mathbb{Z} \subseteq \sum_n (\pm (S \cup \{z\}))\) for some \(b\), \(n> 0\). This means that for every \(y\in \mathbb{Z}\), there are \(k_1\), \dots, \(k_m\), \(k\in \mathbb{Z}\) as well as \mbox{\(x_1\), \dots, \(x_{m}\in S\)} such that \(|k_1| + \dots + |k_m| + |k| \le n\) and
  \[by = k_1x_1 + \dots + k_{m}x_{m} + kz.\]
  The conclusion clearly holds for \(S \subseteq \{0\}\), so suppose there is \(s\in S \setminus \{0\}\). Then the above equation is equivalent to
  \[(bs) \cdot y = (k_1s) \cdot x_1 + \dots + (k_ms)\cdot x_m + (kz)\cdot s.\]
  Setting \(k_i' := k_is\) for \(i\in \{1, \dots, m\}\) and \(k_{m+1}' := kz\) as well as \(x_{m+1} := s\), this shows that for every \(y\in \mathbb{Z}\), there are \(k_1'\), \dots, \(k_{m+1}'\in \mathbb{Z}\), \(x_1\), \dots, \(x_{m+1}\in S\) such that
  \[(bs)\cdot y = k_1'x_1 + \dots + k_{m+1}' x_{m+1}.\]
  So with \(N := n\cdot \max(|s|, |z|)\), we have \(bs \mathbb{Z} \subseteq \sum_{N}(\pm S)\), which means that \(S\) is not sufficiently sparse.
\end{proof}

\begin{proposition}\label{thm:small-sparse}
  A set \(S \subseteq \mathbb{Z}\) is small in \((\mathbb{Z},+,-,0,1)\) if and only if it is sufficiently sparse.
\end{proposition}
\begin{proof}
  Suppose \(S\) is small. Let \(f \colon \mathbb{Z}^m \overset{n}{\to} \mathbb{Z}\) be the multivalued function given by
  \[f(x_{1},\dots,x_{m}) := \left\{ \sum_{i=1}^k \varepsilon_ix_i : k\le n, \varepsilon_i\in\{1,-1\}\right\}.\]
  This is clearly definable in the multivalued sense, and \(f(S^m) = \sum_n(\pm S)\). The smallness of \(S\) implies that \(\sum_{n}(\pm S)\) is also small. Since every non-trivial subgroup of \(\mathbb{Z}\) is large, the set \(\sum_n(\pm S)\) cannot contain any of them. Hence, \(S\) is sufficiently sparse.

  Now suppose that \(S\) is large. Then by Lemma \ref{thm:large-qe}, there are \(a_{i,j}\), \(b_i\), \(c_i \in \mathbb{Z}\) with \(b_i\ne 0\) for \(i\in \{1, \dots, n\}\), \(j\in \{1, \dots, m\}\) such that \(\mathbb{Z} = \bigcup_{i=1}^n S_i\), where
  \[S_i = \left\{ y\in \mathbb{Z} : \exists x_1, \dots, x_m \in S\colon  \sum_{j=1}^{m}a_{i,j} x_{j} = b_{i}y + c_{i}\right\}.\]
  Let \(d_i := |c_i| + \sum_{j=1}^m |a_{i,j}|\) and \(S' := S \cup \{1\}\). Then \(b_i S_i \subseteq \sum_{d_i}(\pm S')\) for every \(i \in \{1, \dots, n\}\). So with \(b := \prod_{i=1}^n b_{i} \) and \(d := |b| \cdot \max_i d_i\), it follows that
    \[b \mathbb{Z} = \bigcup_{i=1}^n b S_i \subseteq \bigcup_{i=1}^n \frac{b}{b_i} \sum\nolimits_{d_i} \!\left(\pm S'\right) \subseteq \sum\nolimits_{d} \!\left(\pm S'\right).\]
  Hence, \(S'\) is not sufficiently sparse. By Lemma \ref{thm:one-more-sparse}, \(S\) is not sufficiently sparse either.
\end{proof}

\begin{corollary}\label{thm:mann-small-z} 
  If \(S \subseteq \mathbb{Z}\) has the Mann property, then \(S\) is small in \((\mathbb{Z},+,-,0,1)\). In particular, the set \(\bigcup_{i=1}^m \ell_i^{\mathbb{N}}\) is small in \((\mathbb{Z},+,-,0,1)\) for any \(\ell_1\), \dots, \(\ell_m\in \mathbb{Z}\).
\end{corollary}
\begin{proof}
  By \cite[Corollary 5.7]{vdDGunaydin}, the Mann property implies the following uniform version of the Mann property, called the \emph{weighted sum property} in \cite[Definition 3.2]{Conant}: For all \(k\ge 1\), there is some \(n\ge 0\) such that, for any \(r\in \mathbb{Z} \setminus \{0\}\), the equation \(x_1 + \dots + x_k = r\) has at most \(n\) non-degenerate solutions in \(S\). Therefore, \(S\) is sufficiently sparse by \cite[Lemma 3.3]{Conant}, and thus small in \((\mathbb{Z},+,-,0,1)\) by Proposition \ref{thm:small-sparse}.
\end{proof}

The advantage of the notion of smallness is that it also makes sense over models whose domain is larger than \(\mathbb{Z}\). We show that in this case, smallness is implied by the Mann axioms.
\begin{lemma} \label{thm:at-most-one}
   Let \(S \subseteq \mathbb{Z}\), and let \(\mathcal{M} = (M,+,-,0,1)\) be a torsion-free abelian group. Furthermore, let \(A \subseteq M\) be such that for all \(a_1\), \dots, \(a_m\), \(b \in \mathbb{Z}\) and \(\alpha_1\), \dots, \(\alpha_m \in A\), if \((\alpha_1, \dots, \alpha_m)\) is a non-degenerate solution of \(a_1x_1+ \dots + a_mx_m = b\), then \(\alpha_1\), \dots, \(\alpha_m \in S\). Then for any \(c \in M\), the set \[\left\langle A \setminus \mathbb{Z} \right\rangle\div \cap (c + \mathbb{Z})\] contains at most one element.
\end{lemma}
\begin{proof}
  Suppose
  \begin{align*}
    a_1 \alpha_1 + \dots + a_m \alpha_m &= b (c+k)\\
    a_1' \alpha_1' + \dots + a_n' \alpha_n' &= b' (c+k')
  \end{align*}
  for some \(a_1\), \dots, \(a_m\), \(a_1'\), \dots, \(a_n'\), \(b\), \(b'\) \(k\), \(k' \in \mathbb{Z}\) with \(b\), \(b' \ne 0\) and \(k \ne k'\), as well as some \(\alpha_1\), \dots, \(\alpha_m\), \(\alpha'_1\), \dots, \(\alpha'_n \in A \setminus \mathbb{Z}\). Then subtracting \(b'\) times the first equation from \(b\) times the second equation gives
  \[\sum_{j=1}^n a'_jb \alpha'_j - \sum_{i=1}^m a_ib' \alpha_i = bb'(k'-k) \in \mathbb{Z} \setminus \{0\}.\]
  This contradicts our assumption (after removing all parts that sum to zero to obtain a non-degenerate solution).
\end{proof}

\begin{proposition}\label{thm:mann-small-general} 
  Let \(S \subseteq \mathbb{Z}\) satisfy the Mann property, and let \(\mathcal{M}\) be a model of \(\Th(\mathbb{Z},+,-,0,1)\). Furthermore, let \(A \subseteq M\) be such that for all \linebreak[2]\(a_1\), \dots, \(a_m\), \(b \in \mathbb{Z}\) and \(\alpha_1\), \dots, \(\alpha_m \in A\), if \((\alpha_1, \dots, \alpha_m)\) is a non-degenerate solution of \(a_1x_1+ \dots + a_mx_m = b\), then \(\alpha_1\), \dots, \(\alpha_m \in S\). Then \(A\) is small in \(\mathcal{M}\).
\end{proposition}
\begin{proof}
  If \(A\) is large in \(\mathcal{M}\), then by Lemma \ref{thm:large-qe}, there are \(a_{ij}\in \mathbb{Z}\), \(b_i\in \mathbb{Z} \setminus \{0\}\) and \(c_i \in M\) such that
  \[M = \bigcup_{i=1}^n \left\{y\in M : \exists x_1, \dots, x_{m_i}\in A \colon  \sum_{j=1}^{m_i} a_{ij}x_j = b_iy + c_i\right\}.\]
  In particular,
  \[\mathbb{Z} = \bigcup_{i=1}^n \left\{y\in \mathbb{Z} : \exists x_1, \dots, x_{m_i}\in A \colon  \sum_{j=1}^{m_i} a_{ij}x_j = b_iy + c_i\right\}.\]
  To generate an integer \(y\), we need that \(\sum_j a_{ij} x_j \in c_i + \mathbb{Z}\). Set
  \[I := \left\{i\in \{1, \dots, n\} : \left\langle A \setminus \mathbb{Z} \right\rangle\div \cap (c_i + \mathbb{Z}) \ne \emptyset \right\}.\]
  For each \(i\in I\), let \(s_i\) be the unique element of \(\left\langle A \setminus \mathbb{Z} \right\rangle\div \cap (c_i + \mathbb{Z}) \) (cf. Lemma \ref{thm:at-most-one}). This means that for all \(y \in \mathbb{Z}\), if
  \[\sum_{j=1}^{m_i} a_{ij} \alpha_j = b_iy + c_i\]
  for some \(\alpha_1\), \dots, \(\alpha_{m_i} \in A\), then there is a subset \(J \subseteq \{1, \dots, m_i\}\), namely \(J = \{j : \alpha_j \in \mathbb{Z}\}\), such that
  \[\sum_{j\notin J} a_{ij} \alpha_j = s_i \quad \text{and} \quad  \sum_{j\in J} a_{ij} \alpha_j = b_iy + c_i - s_i.\]
  So we have shown that
  \[\mathbb{Z} = \bigcup_{i, J} \left\{y\in \mathbb{Z} : \exists \overline{x} \in S^{|J|} \colon  \sum_{j \in J} a_{ij}x_j = b_iy + (c_i - s_i)\right\}\]
  where the union ranges over all \(i\in I\) and all \(J \subseteq \{1, \dots, m_i\}\). Since \(c_i - s_i \in \mathbb{Z}\) for all \(i \in I\), this contradicts that \(S\) is small in \((\mathbb{Z},+,-,0,1)\) by Corollary \ref{thm:mann-small-z}.
\end{proof}

\section{Complete Theory Without Order}
\label{ch:orderless}

Throughout this section, fix a set \(L \subseteq \mathbb{Z}_{\ge 2}\) of pairwise multiplicatively independent numbers.

\begin{definition} \label{def:theory-without-order}
  Let \(T(L)\) be the \(\mathcal{L}\)-theory containing the following axiom schemata:
  \begin{enumerate}[label=(A\arabic*), start=1]
    \item the axioms of torsion-free abelian groups, as well as \(0 \ne 1\),
    \item the congruence axioms of Presburger:
      \[\forall x \,\bigvee_{j=1}^n \Biggl(D_n(x+j) \wedge \bigwedge_{\substack{k=1 \\k\ne j}}^{n} \neg D_n(x+k)\Biggr) \quad \text{ for every } n\ge 2,\] \label{ax:congruence} \vspace{-3mm}
    \item \(U_{\ell}(1) \wedge \forall x (U_\ell(x) \leftrightarrow U_\ell(\ell\cdot x))\) for each \(\ell\in L\), \label{ax:multiplication}
    \item the Mann axioms \(\Mann(L)\) (see Definition \ref{def:mann-axioms}),
    \item the Carmichael axioms \(\Carmichael(L)\) (see Definition \ref{def:carmichael-axioms}).
  \end{enumerate}
\end{definition}
As explained in Section \ref{ch:axioms}, we have \((\mathbb{Z},+,-,0,1,(\ell^{\mathbb{N}})_{\ell\in L}) \models T(L)\). 

\subsection{Back and Forth} We now show that this axiomatization is complete by constructing a back-and-forth system.

\begin{lemma} \label{thm:partial-type-realizable}
  Let \(\mathcal{M}\), \(\mathcal{N} \models T(L)\) be \(\kappa\)-saturated for some uncountable cardinal \(\kappa\). Furthermore, let \(M' \subseteq M\) and \(N' \subseteq N\) be subsets of cardinality less than \(\kappa\), and let \(f\colon M' \to N' \) be a partial isomorphism. Then for any \(\ell\in L\) and any \(\alpha\in A_\ell \setminus \mathbb{Z}\), the partial type
  \[p(x) := d_{\mathcal{M}}(\alpha) \cup \{U_\ell(x)\} \cup \{x\ne f(g) \colon g\in M'\}\]
  is realized in \(\mathcal{N}\).
\end{lemma}
\begin{proof}
  Since \(\mathcal{N}\) is \(\kappa\)-saturated and \(p(x)\) is a type with \(|M'| < \kappa\) parameters, it suffices to show that \(p(x)\) is finitely satisfiable. Let \(\Delta\) be a finite subset of \(d_{\mathcal{M}}(g)\), and let \(\Psi\) be a finite subset of \(p(x)\) with \(\Psi \cap d_{\mathcal{M}}(\alpha) = \emptyset \). If there exist infinitely many realizations \((\alpha_i)_{i\in \mathbb{N}}\) of \(\Delta\) in \(\ell^{\mathbb{N}} \subseteq \mathcal{N}\), then since \(\Psi\) contains only finitely many formulas of the form \(x\ne f(g)\), all but finitely many \(\alpha_i\) must satisfy \(\Psi\). It thus suffices to show that every finite \(\Delta \subseteq d_{\mathcal{M}}(\alpha)\) is satisfied by infinitely many elements of \(\ell^{\mathbb{N}}\). This follows from the Carmichael axioms by Lemma \ref{thm:carmichael-use}.
\end{proof}

\begin{lemma} \label{thm:congruence-outside-subgroup}
  Let \(\mathcal{M} \models \Th(\mathbb{Z},+)\), let \(\kappa\) be an uncountable cardinal, and let \((\mathcal{N},(B_\ell)_{\ell\in L}) \models T(L)\) be \(\kappa\)-saturated. Let \(\alpha\in M\) and let \(N' \subseteq N\) be a subset of cardinality less than \(\kappa\). Then there exists \(\beta\in N \setminus \left\langle N', B \right\rangle\div\) realizing the congruence type \(d_{\mathcal{M}}(\alpha)\).
\end{lemma}
\begin{proof}
  Let \(p(x)\) be the type consisting of the formulas
  \[\nexists y_1, \dots, y_n \left( \bigwedge_{i=1}^n U_{\ell_i}(y_i) \wedge a_1y_1 + \dots + a_ny_n = bx + c\right)\]
  for all \(n\ge 1\), \(\ell_1\), \dots, \(\ell_n\in L\), \(a_1\), \dots, \(a_n\), \(b\in \mathbb{Z}\setminus \{0\}\) and \(c\in N'\). Then \(N \setminus \left\langle N', B \right\rangle\div \) is exactly the set of realizations of \(p(x)\) in \(\mathcal{N}\). Every finite subset of \(p(x)\) is realizable in \(\mathcal{N}\) by Proposition \ref{thm:mann-small-general} and Lemma \ref{thm:large-qe}. So by saturation, \(p(x)\) is realized in \(\mathcal{N}\), which shows that \(N \setminus \left\langle N', B \right\rangle\div \ne \emptyset \).

  We want to find a realization \(\beta\in N\) of \(d_{\mathcal{M}}(\alpha) \cup p(x)\). Let \(\Delta \subseteq d_{\mathcal{M}}(\alpha)\) and \(\Psi \subseteq p(x)\) be finite. Take \(\beta' \in N \setminus \left\langle N', B \right\rangle\div\). Since \(\left\langle N', B \right\rangle\div\) is a proper subgroup of \(N\) containing \(1\), we have \(\beta' + k \notin \left\langle N', B \right\rangle\div\) for every \(k\in \mathbb{Z}\), so in other words, \(\beta' + k\) satisfies \(\Psi\) for every \(k\in \mathbb{Z}\). By the congruence axioms \ref{ax:congruence}, there is some \(k\in \mathbb{Z}\) such that \(\beta'+k\) satisfies \(\Delta\). This shows that \(d_{\mathcal{M}}(\alpha) \cup p(x)\) is finitely satisfiable and thus realized by saturation.
\end{proof}

\begin{definition} \label{def:sub}
  Let \((\mathcal{M}, (A_{\ell})_{\ell\in L}) \models T(L)\) and let \(\kappa\) be an uncountable cardinal. Then we define \(\Sub_{\kappa}(\mathcal{M})\) as the set of all subgroups \(M'\) of \(\mathcal{M}\) that satisfy the following properties:
  \begin{enumerate}
    \item \(1 \in M'\). \label{it:sub1}
    \item \(|M'| < \kappa\).  \label{it:sub2}
    \item The group \(M'\) is divisibly closed in \(\mathcal{M}\). \label{it:sub3}
    \item For any \(a_1\), \dots, \(a_m \in A \setminus M'\) and \(c_1\), \dots, \(c_m\in \mathbb{Z}\), we have that \(\sum_{i=1}^m c_ia_i \in M'\) implies \(\sum_{i=1}^m c_ia_i = 0\). \label{it:sub4}
  \end{enumerate}
\end{definition}

\begin{proposition} \label{thm:back-and-forth-system}
  Let \(\kappa\) be an uncountable cardinal, and let \((\mathcal{M},(A_\ell)_{\ell\in L})\), \((\mathcal{N},(B_\ell)_{\ell\in L})\) be \(\kappa\)-saturated models of \(T(L)\). Then
  \[\mathcal{I} := \left\{ f \colon M' \to N' \text{ partial isomorphism}: M'\in \Sub_{\kappa}(\mathcal{M}),\, N'\in \Sub_{\kappa}(\mathcal{N})\right\} \]
  is a non-empty back-and-forth system.
\end{proposition}
\begin{proof}
  We first show that \(\mathcal{I} \ne \emptyset \). We have \(\mathbb{Z}\in \Sub_{\kappa}(\mathcal{M})\) and \(\mathbb{Z}\in \Sub_{\kappa}(\mathcal{N})\), where the properties \ref{it:sub1}--\ref{it:sub3} of Definition \ref{def:sub} are clear, and property \ref{it:sub4} follows from the Mann axioms. Therefore, the identity \(\id \colon \mathbb{Z} \to \mathbb{Z}\) is an element of \(\mathcal{I}\).

  We now give a proof of the forth direction; the back direction is proven analogously. Let \(M'\in \Sub_{\kappa}(\mathcal{M})\) and \(N' \in \Sub_{\kappa}(\mathcal{N})\), and let \(f\colon M' \to N' \) be a partial isomorphism in \(\mathcal{I}\). Let \(\alpha\in M \setminus M'\). Our goal is to find \(\beta \in N \setminus N'\) such that \(F\) defined as follows is a partial isomorphism in \(\mathcal{I}\):
  \begin{align*}
    F \colon \left\langle M', \alpha \right\rangle\div &\to \left\langle N', \beta\right\rangle\div\\
    \frac{g + k\alpha}{n} &\mapsto \frac{f(g)+k\beta}{n}  \qquad \text{for } g\in M', k,n \in \mathbb{Z}, n>0.
  \end{align*}
  This map is well-defined and surjective as long as \(f(g)+k\beta\) is divisible by \(n\) in \(\mathcal{N}\) if and only if \(g+k\alpha\) is divisible by \(n\) in \(\mathcal{M}\). Since \(f\) is a partial isomorphism, we have \(d_{\mathcal{M}}(g) = d_{\mathcal{N}}(f(g))\), so it would suffice to have \(d_{\mathcal{M}}(\alpha) = d_{\mathcal{N}}(\beta)\). It is clear that \(F\), if well-defined, is a group homomorphism. It is also injective: Let \(a\) be in the kernel of \(F\). Write \(a = \frac{g+k\alpha}{n}\) for some \(g\in M'\) and \(k,n\in \mathbb{Z}\), \(n>0\). Now
    \[F(a) = \frac{f(g)+k\beta}{n} = 0\]
    implies that \(f(g) = -k\beta\). If \(k\) were non-zero, we would get the contradiction \(\beta \in N'\), since \(f(g)\in N'\) and \(N'\) is divisibly closed in \(N\). So \(k=0\), which implies \(f(g)=0\) and thus \(g=0\), since \(f\) is a partial isomorphism. So in total, we have \(a=0\), which shows that \(F\) is injective and thus a group isomorphism.

  For the remainder of our task, we consider three cases:

  \begin{case} \(\alpha\in A_\ell\) for some \(\ell\in L\).

    By Lemma~\ref{thm:partial-type-realizable}, we can find \(\beta\in B_{\ell}\setminus N'\) such that \(d_{\mathcal{M}}(\alpha) = d_{\mathcal{N}}(\beta)\). Let \(a\in \left\langle M', \alpha \right\rangle\div\), which we write as \(a = \frac{g+k\alpha}{n}\), as before. Let \(j\in L\). It remains to show that \(a \in A_{j}\) if and only if \(F(a) \in B_j\). We only need to show one implication, as the other one works analogously by replacing \(F\) with \(F^{-1}\). So suppose \(a\in A_j\). If \(k=0\), then \(a\in M'\), so the claim follows, since \(f\) is a partial isomorphism. Now suppose \(k\ne 0\). Then \(k\alpha - na = -g \in M'\), whereas \(\alpha\), \(a\in A \setminus M'\). So by Definition \ref{def:sub}~\ref{it:sub4}, we have \(g=0\) and thus \(na = k\alpha\). Now if \(j \ne \ell\), then the Mann axioms imply that \(a\) and \(\alpha\) are some of the hard-coded solutions in \(j^{\mathbb{N}}\) and \(\ell^{\mathbb{N}}\), which contradicts \(a\), \(\alpha \notin \mathbb{Z}\). So \(j = \ell\), in which case the Mann axioms imply that there are \(s\), \(t\in \ell^{\mathbb{N}}\) such that \(s a = t \alpha\). Therefore, \(F(a) = \frac{t}{s} \beta \in B_{\ell}\) follows from \ref{ax:multiplication}.

    So we have shown that \(F\) is a partial isomorphism. Finally, in order to show that \(F\in \mathcal{I}\), we need to verify that \(\left\langle M', \alpha \right\rangle\div \in \Sub_{\kappa}(\mathcal{M})\). The conditions \ref{it:sub1}--\ref{it:sub3} of Definition \ref{def:sub} are clear. In order to verify condition \ref{it:sub4}, let \(a_1\), \dots, \(a_m\in A \setminus \left\langle M', \alpha \right\rangle\div\) and \(c_1\), \dots, \(c_m\in \mathbb{Z}\), where we can without loss of generality assume \(\sum_{i \in I} c_ia_i \ne 0\) for all non-empty \(I \subseteq \{1, \dots, m\}\). Suppose that \(\sum_{i=1}^m c_i a_i \in \left\langle M', \alpha \right\rangle\div\), so write \(\sum_{i=1}^m c_ia_i = \frac{g + k\alpha}{n}\) for some \(g\in M'\) and \(k\), \(n \in \mathbb{Z}\) with \(n\ne 0\). Rearranging yields
    \[ \left(\sum_{i=1}^m nc_ia_i\right) - k \alpha = g.\]
    Since \(g\in M'\) and \(\alpha \notin M'\), we can apply condition \ref{it:sub4} of Definition \ref{def:sub} to \(M' \in \Sub_{\kappa}(\mathcal{M})\) to obtain \(g=0\), so
    \[\sum_{i=1}^m nc_i a_i = k\alpha.\]
    If \(k\ne 0\), then the Mann axioms imply that for all \(i\in \{1, \dots, m\}\), there are \(s, t \in \mathbb{Z}\) such that \(ta_i = s\alpha\). So we have \(a_i = \frac{s}{t} \alpha \in \left\langle M', \alpha \right\rangle\div \), contradicting our assumptions. Thus, \(k=0\) and \(\sum_{i=1}^m c_ia_i = 0\), as desired.
  \end{case}

  \begin{case} \(\alpha\in \left\langle M', A \right\rangle\div\).

    This case is a simple generalization of the previous case. Pick \(a_{1}\), \dots, \(a_{m} \in A \setminus M'\) such that \(\alpha \in \left\langle M', a_1, \dots, a_{m} \right\rangle\div\). Now we can apply the first case \(m\) times to \(a_1\), \dots, \(a_m\) to get a partial isomorphism \(F \in \mathcal{I}\) from \(\left\langle M', a_1, \dots, a_{m} \right\rangle\div\) into \(\mathcal{N}\).
  \end{case}

  \begin{case} \(\alpha\notin \left\langle M', A \right\rangle\div\).
    
    By Lemma \ref{thm:congruence-outside-subgroup}, there is \(\beta\in N \setminus \left\langle N',B \right\rangle\div\) with \(d_{\mathcal{M}}(\alpha) = d_{\mathcal{N}}(\beta)\). It remains to show that the resulting \(F\) is in \(\mathcal{I}\). Let \(a\in \left\langle M', \alpha \right\rangle\div\), which we again write as \(a = \frac{g+k\alpha}{n}\). Suppose \(a\in A_{\ell}\). Then if \(k\ne 0\), we get the contradiction \(\alpha = \frac{na-g}{k}\in \left\langle M', A \right\rangle\div\). So \(k=0\), which implies \(a\in M'\) and thus \(F(a) = f(a) \in B_{\ell}\), since \(f\) is a partial isomorphism. The reverse implication that \(F(a)\in B_\ell\) implies \(a\in A_\ell\) again follows similarly by considering \(F^{-1}\).

    It remains to show that \(\left\langle M', \alpha \right\rangle\div \in \Sub_{\kappa}(\mathcal{M})\). The conditions \ref{it:sub1}--\ref{it:sub3} of Definition \ref{def:sub} are again clear. To verify \ref{it:sub4}, let \(a_1\), \dots, \(a_m \in A \setminus \left\langle M', \alpha \right\rangle\div\) and \(c_1\), \dots, \(c_m\in \mathbb{Z}\) be such that \(\sum_{j=1}^m c_ja_j\in \left\langle M',\alpha \right\rangle\div\). This means there are \(g\in M'\) and \(k\), \(n\in \mathbb{Z}\), \(n\ne 0\) such that \(\sum_{j=1}^{m} c_{j}a_{j} = \frac{g+k\alpha}{n}\). Again, \(k \ne 0\) would imply \[\alpha = \frac{-ng + \sum_{j=1}^m c_ja_j}{k}\in \left\langle M',A \right\rangle\div,\] which is a contradiction. So \(k=0\) and thus \(\sum_{j=1}^m c_ja_j \in M'\), which implies \(\sum_{j=1}^m c_ja_j=0\), as \(M'\in \Sub_{\kappa}(\mathcal{M})\).
  \end{case}\vspace{-1.8em}
\end{proof}

\begin{corollary} \label{thm:complete}
  The theory \(T(L)\) is complete.
\end{corollary}
\begin{proof}
  Let \(\mathcal{M}\) and \(\mathcal{N}\) be \(\kappa\)-saturated models of \(T(L)\). By Proposition \ref{thm:back-and-forth-system}, there is a back-and-forth system between \(\mathcal{M}\) and \(\mathcal{N}\), which implies that \(\mathcal{M}\) and \(\mathcal{N}\) are elementarily equivalent.

  Now let \(\mathcal{M}'\), \(\mathcal{N}' \models T(L)\) be arbitrary. Then there exist \(\kappa\)-saturated \(\mathcal{M}, \mathcal{N} \models T(L)\) such that \(\mathcal{M}'\) is an elementary substructure of \(\mathcal{M}\) and \(\mathcal{N}'\) is an elementary substructure of \(\mathcal{N}\). Since \(\mathcal{M}\) and \(\mathcal{N}\) are elementarily equivalent, so are \(\mathcal{M}'\) and \(\mathcal{N}'\).
\end{proof}

\begin{corollary} \label{thm:decidable}
  Let \(k\), \(\ell\in \mathbb{Z}_{\ge 2}\) be multiplicatively independent. Then the structure \((\mathbb{Z},+,k^{\mathbb{N}}, \ell^{\mathbb{N}})\) is decidable.
\end{corollary}
\begin{proof}
  By Corollary \ref{thm:complete}, the theory \(T(\{k, \ell\})\) completely axiomatizes the theory of \((\mathbb{Z},+,-,0,1,k^{\mathbb{N}}, \ell^{\mathbb{N}})\). As seen in Section \ref{ch:axioms} (Lemma \ref{thm:nondegenerate-decidability} in particular), the theory \(T(\{k, \ell\})\) is also computable. Therefore, \((\mathbb{Z},+,-,0,1,k^{\mathbb{N}}, \ell^{\mathbb{N}})\) is decidable.
\end{proof}

\subsection{Quantifier Elimination}
Let \(\Sigma\) be the set of all boolean combinations of \(\mathcal{L}\)-formulas of the form
\[\exists y \left( \bigwedge_{i=1}^n U_{\ell_i}(y_i) \wedge \varphi(x,y)\right),\]
where \(x\) and \(y = (y_1, \dots, y_n)\) are tuples of variables, \(\ell_1\), \dots, \(\ell_n \in L\), and \(\varphi(x,y)\) is a quantifier-free formula in the language \(\{+,-,0,1\} \cup \{D_k : k \in \mathbb{N}\}\).

\begin{theorem} \label{thm:qe}
  Modulo \(T(L)\), every \(\mathcal{L}\)-formula is equivalent to a formula in~\(\Sigma\).
\end{theorem}

\begin{proof}
  Let \(\kappa\) be an uncountable cardinal, let \((\mathcal{M},(A_{\ell})_{\ell\in L})\), \((\mathcal{N}, (B_{\ell})_{\ell\in L})\) be \(\kappa\)-saturated models of \(T(L)\), and let \(\alpha = (\alpha_1, \dots, \alpha_m) \in M^m\) and \(\beta = (\beta_1, \dots, \beta_m) \in N^m\) be tuples that satisfy the same formulas from \(\Sigma\). We need to show that \(\alpha\) and \(\beta\) have the same parameter-free type. By Proposition \ref{thm:back-and-forth-system}, it suffices to find \(M' \in \Sub_{\kappa}(\mathcal{M})\), \(N' \in \Sub_{\kappa}(\mathcal{N})\) and a partial isomorphism \(f \colon M' \to N' \) such that \(\alpha_i \in M'\), \(\beta_i \in  N'\) and \(f(\alpha_i) = \beta_i\) for every \(i \in \{1, \dots, m\}\).

  By re-ordering the tuples \(\alpha\) and \(\beta\), we may assume that the first \(r\) coordinates \(\{\alpha_1, \dots, \alpha_r\}\) form a maximal subset of \(\{\alpha_1, \dots, \alpha_m\}\) that is linearly independent over \(A\), where \(r \in \{0, \dots, m\}\). Now \(\{\beta_1, \dots, \beta_r\}\) is a maximal subset of \(\{\beta_1, \dots, \beta_m\}\) that is linearly independent over \(B\), as this can be expressed using formulas in \(\Sigma\). Let \(a = (a_1, \dots, a_n) \in A^n\) be such that
  \[1, \alpha_1, \dots, \alpha_r, a_1, \dots, a_n \text{ are linearly independent}\]
  and for each \(i \in \{r+1, \dots, m\}\), we have
  \[\alpha_i = \frac{z_i + c_{i1}\alpha_1 + \dots + c_{ir}\alpha_r + d_{i1}a_1 + \dots d_{i n}a_n}{e_i},\]
  where \(z_i\), \(c_{i1}\), \dots, \(c_{ir}\), \(d_{i1}\), \dots, \(d_{i n} \in \mathbb{Z}\), and \(e_i \in \mathbb{Z} \setminus \{0\}\). 
  Let \(p(y)\) be the type consisting of \(d_{\mathcal{M}}(a)\) together with the formulas
  \begin{align*}
    \beta_i &= \frac{z_i+c_{i1}\beta_1 + \dots + c_{ir}\beta_r + d_{i1}y_1 + \dots d_{i n}y_n}{e_i} \qquad \text{and}\\
    0 &\ne z+s_1\beta_1 + \dots + s_r\beta_r + t_1y_1+ \dots + t_ny_n
  \end{align*}
  for every \(i \in \{r+1, \dots, m\}\) and every \((z,s, t) \in \mathbb{Z}^{1+r+n} \setminus \{0\}\).
  Furthermore, let \(\ell_1\), \dots, \(\ell_n \in L\) be such that \(a_i \in A_{\ell_i}\) for every \(i \in \{1, \dots, n\}\), and let \(\theta(y)\) be the formula \(\bigwedge_{i=1}^n U_{\ell_i}(y_i)\). Now if \(\varphi(\beta, y)\) is a conjunction of finitely many formulas in \(p(y)\), then the formula \(\psi(x)\) defined as
  \[\exists y \big( \theta(y) \wedge \varphi(x,y)\big) \]
  is in \(\Sigma\). So since \(\mathcal{M} \models \varphi(\alpha,a)\), we have \(\mathcal{M} \models \psi(\alpha)\) and thus \(\mathcal{N} \models \psi(\beta)\). This shows that the type \(p(y) \cup \{\theta(y)\}\) is finitely satisfiable in \(\mathcal{N}\). By saturation, it has a realization \(b = (b_1, \dots, b_n) \in B^n\).

  Now let 
  \begin{align*}
    M' &:= \left\langle 1, \alpha_1, \dots, \alpha_r, a_1, \dots, a_n \right\rangle\div,\\
    N' &:= \left\langle 1, \beta_1, \dots, \beta_r, b_1, \dots, b_n \right\rangle\div,
  \end{align*}
  and define the group homomorphism \(f \colon M' \to N'\) by \(f(1)=1\), \(f(\alpha_i) = \beta_i\) for \(i \in \{1, \dots, r\}\) and \(f(a_i) = b_i\) for \(i \in \{1, \dots, n\}\). Since \(b\) satisfies \(p(y)\), we also have
  \[\beta_i = \frac{z_i+c_{i1}\beta_1 + \dots + c_{ir}\beta_r + d_{i1}b_1 + \dots d_{i n}b_n}{e_i}\]
  and thus \(f(\alpha_i)=\beta_i\) for \(i \in \{r+1, \dots, m\}\). Since both \[(1, \alpha_1, \dots, \alpha_r, a_1, \dots, a_n) \quad\text{and}\quad (1,\beta_1, \dots, \beta_r, b_1, \dots, b_n)\] are linearly independent, \(f\) is a group isomorphism. We now show that \(f\) is also an \(\mathcal{L}\)-homomorphism. Let \(g \in M'\), which we write as
  \[g = \frac{z+u_1\alpha_1 + \dots + u_r\alpha_r + v_1a_1 + \dots + v_na_n}{w}\]
  with \(z\), \(u_1\), \dots, \(u_r\), \(v_1\), \dots, \(v_n \in \mathbb{Z}\) and \(w \in \mathbb{Z} \setminus \{0\}\). If \(g \in A_{\ell}\) for some \(\ell \in L\), then
  \[u_1\alpha_1 + \dots + u_r\alpha_r = wg - z - v_1a_1 - \dots - v_na_n \in \left\langle A \right\rangle.\]
  Since \(\alpha_1\), \dots, \(\alpha_r\) are linearly independent over \(A\), we have \(u_1 = \dots = u_r = 0\) and thus \(wg = z+v_1a_1 + \dots + v_na_n\). If \(g \in \mathbb{Z}\), then \(f(g) = g \in B_{\ell}\), so we are done. Otherwise, we must have \(v_i \ne 0\) for some \(i \in \{1, \dots, n\}\). Applying the Mann axioms yields \(g = \ell^t a_i\) for some \(t \in \mathbb{Z}\) and thus \(f(g) = \ell^t b_i\). Hence, \ref{ax:multiplication} implies \(f(g) \in B_l\), as desired. One proves that \(g\notin A_\ell\) implies \(f(g)\notin B_\ell\) by repeating the same argument for \(f^{-1}\) instead of \(f\).

  Finally, we show that \(M' \in \Sub_{\kappa}(\mathcal{M})\). The first three properties of Definition~\ref{def:sub} are clear. Proving \ref{it:sub4} works similarly as in the previous argument: If \(g_1\), \dots, \(g_k \in A \setminus M'\) and \(s_1\), \dots, \(s_k \in \mathbb{Z}\), then \(s_1g_1 + \dots + s_kg_k \in M'\) means that
  \[s_1g_1 + \dots + s_kg_k = \frac{z+u_1\alpha_1 + \dots + u_r\alpha_r + v_1a_1 + \dots + v_na_n}{w}\]
  for some \(z\), \(u_1\), \dots, \(u_r\), \(v_1\), \dots, \(v_n \in \mathbb{Z}\) and \(w \in \mathbb{Z} \setminus \{0\}\). Since \(\alpha_1\), \dots, \(\alpha_r\) are linearly independent over \(A\), we again obtain \(u_1 = \dots = u_r = 0\). If \(s_j \ne 0\), then applying the Mann axioms yields \(g_j \in M'\). So we must have \(s_j = 0\) for all \(j \in \{1, \dots, k\}\). One analogously shows \(N' \in \Sub_{\kappa}(\mathcal{N})\), which finishes the proof.
\end{proof}

\begin{corollary} \label{thm:superstable}
  The theory of \((\mathbb{Z},+,(\ell^{\mathbb{N}})_{\ell \ge 2})\) is superstable.
\end{corollary}
\begin{proof}
  By Lemma \ref{thm:dependence} and Corollary \ref{thm:complete}, it suffices to show that \(T(L)\) is superstable for a maximal \(L \subseteq \mathbb{Z}_{\ge 2}\) whose elements are pairwise multiplicatively independent. Let \(\kappa \ge 2^{\aleph_0}\) and let \((\mathcal{M}, (A_{\ell})_{\ell \in L}) \models T(L)\) with \(|M| = \kappa\). We need to show that there are only \(\kappa\) complete \(1\)-types over \(\mathcal{M}\). Let \((\mathcal{N}, (B_{\ell})_{\ell \in L})\) be a \(\kappa^+\)-saturated elementary extension of \(\mathcal{M}\), and let \(\alpha \in N\). By Theorem \ref{thm:qe}, the type \(\tp(\alpha/M)\) is completely determined by \(d_{\mathcal{N}}(\alpha) = d_{\mathcal{M}}(\alpha)\) together with the set of all tuples \((u, \overline{c}, d, \overline{\ell})\) such that \(u \in M\), \(n \in \mathbb{N}\), \(\overline{c} = (c_1, \dots, c_n) \in \mathbb{Z}^n\), \(d \in \mathbb{Z} \setminus \{0\}\), \(\overline{\ell} = (\ell_1, \dots, \ell_n) \in L^n\) and
  \begin{equation}
  \label{eq:5}
  \exists b_1 \in B_{\ell_1}, \dots, b_n \in B_{\ell_n} \left(\alpha = \frac{u + c_1b_1 + \dots + c_nb_n}{d}\right).
  \end{equation}
  Fix such \(u\), \(\overline{c}\), \(d\) and \(\overline{\ell}\). Now for every \(v \in M\), we have
  \[\exists b_1' \in B_{\ell_1}, \dots, b_n' \in B_{\ell_n} \left(\alpha = \frac{v + c_1b_1' + \dots + c_nb_n'}{d}\right) \]
  if and only if
  \[\exists b_1, b_1' \in B_{\ell_1}, \dots, b_n, b_n' \in B_{\ell_n} \Bigl(v-u = c_1(b_1-b_1') + \dots + c_n(b_n - b_n')\Bigr).\]
  This shows that for given \((\overline{c}, d, \overline{\ell})\), one \(u \in M\) that satisfies \eqref{eq:5} already specifies all such \(u\). Therefore, the number of complete \(1\)-types over \(\mathcal{M}\) is at most the number of congruence types plus the number of \(u \in M\) times the number of tuples \((\overline{c}, d, \overline{\ell})\) as above, so at most
  \[2^{\aleph_0} + \kappa \cdot \aleph_{0} = \kappa. \qedhere\]
\end{proof}

\section{Universal Theory With Order}
\label{ch:ordered}

Throughout this section, let \(L \subseteq \mathbb{Z}_{\ge 2}\) be a set of pairwise multiplicatively independent integers, and let \(\mathcal{Z} := (\mathbb{Z},+,-,0,1,<,(\ell^{\mathbb{N}})_{\ell\in L})\).
\begin{definition} \label{def:theory-with-order}
  Let \(T_{\forall}(L)\) be the \(\mathcal{L}_{<}\)-theory containing the following axiom schemata:
  \begin{enumerate}[label=$(\forall\mkern1mu \arabic*)$, start=1]
    \item axioms of ordered abelian groups, \label{ax:oag} 
    \item the universal congruence axioms
      \[\forall x \bigvee_{j=1}^n \bigwedge_{\substack{k=1 \\k\ne j}}^{n} \neg D_n(x+k) \qquad \text{ for every } n\ge 2,\] \label{ax:universal-congruence} \vspace{-2mm}
    \item \(U_{\ell}(1) \wedge \forall x (U_\ell(x) \leftrightarrow U_\ell(\ell\cdot x))\) for each \(\ell\in L\), \label{ax:universal-multiplication}
    \item the discreteness axiom \(0 < 1 \wedge \forall x \left(x\le 0 \vee x \ge 1\right) \), \label{ax:discrete}
    \item the Mann axioms \(\Mann(L)\) (see Definition \ref{def:mann-axioms}),\label{ax:universal-mann} 
    \item the inequality axioms \(\Inequ(L)\) (Definition \ref{def:inequality-axioms}) for \(|L| \ge 3\), and the basic inequality axioms \(\BInequ(L)\) (Definition \ref{def:basic-ineq-axioms}) for \(|L| \le 2\).\label{ax:inequality} 
    \item the Carmichael axioms \(\Carmichael(L)\) for \(|L| \le 2\) (see Definition \ref{def:carmichael-axioms}),\label{ax:universal-carmichael} 
  \end{enumerate}
\end{definition}

Notice that all the axioms of \(T_{\forall }(L)\) are universal and that \(\mathcal{Z} \models T_{\forall }(L)\). Our goal is now to show that \(T_{\forall }(L)\) completely axiomatizes the universal theory of \(\mathcal{Z}\). To do this, we need to show that for every \(\mathcal{M} \models T_{\forall}(L)\), there exists a model \(\mathcal{N} \models \Th(\mathcal{Z})\) and an \(\mathcal{L}_{<}\)-embedding \(\mathcal{M} \hookrightarrow \mathcal{N}\).

\subsection{Solving Inequalities}
\label{sec:solving-inequalities}
We now show that our axioms \(T_{\forall }(L)\) imply \(\Inequ(L)\) even for \(|L| \le 2\). 

\begin{fact}[Kronecker's Approximation Theorem] \label{thm:kronecker} 
  Let \(\lambda \in \mathbb{R}\) be irrational. Then the set
  \[\bigl\{ \fr(n\lambda) : n \in  \mathbb{N}\bigr\} \]
  is dense in \([0,1)\), where \(\fr(x) := x - \left\lfloor x \right\rfloor \) denotes the fractional part of \(x \in \mathbb{R}\). 
\end{fact}

\begin{lemma} \label{thm:kronecker-dense} 
  Let \(k\), \(\ell\in \mathbb{Z}_{\ge 2}\), and let \(\mathcal{I} \subseteq \mathbb{R}_{>0}\) be a non-empty open interval. Then there exists a non-empty open interval \(\mathcal{J} \subseteq (0,1)\) such that for every \(t \in \mathbb{N}\) with \(\fr(t \log_k(\ell)) \in \mathcal{J}\), there exists \(s \in  \mathbb{N}\) satisfying \(k^s/\ell^t \in \mathcal{I}\). In particular, if \(k\) and \(\ell\) are multiplicatively independent, then the set
  \[\left\{ \frac{k^s}{\ell^t}  : s, t\in \mathbb{N}\right\} \]
  is dense in \(\mathbb{R}_{>0}\).
\end{lemma}
\begin{proof}
  Shrinking \(\mathcal{I}\) if necessary, we may assume that \(a := \inf\mathcal{I} > 0\). Notice that \(k^s/\ell^t \in \mathcal{I}\) is equivalent to \(s - t \log_k(\ell) \in \log_k(\mathcal{I})\), so we have \(\fr(t \log_k(\ell)) \in \fr(-\log_k(\mathcal{I}))\) if and only if there exists \(s \in \mathbb{Z}\) with \(k^s/\ell^t \in \mathcal{I}\). It remains to find \(\mathcal{J} \subseteq \fr(-\log_k(\mathcal{I}))\) such that the \(s \in \mathbb{Z}\) we obtain is non-negative. Let \(t_0 \in \mathbb{N}\) be greater than \(-\log_k(a)\). Now if \(s < 0\) and \(t>t_0\), then \(k^{s}/\ell^t < 1/\ell^{t_0} < a \). So it suffices to pick a non-empty open \(\mathcal{J} \subseteq \fr(-\log_k(\mathcal{I}))\) which does not contain the finitely many points \(\fr(t\log_k(\ell))\) for \(t \in \{0, \dots, t_0\}\).

  For the second claim, notice that the multiplicative independence of \(k\) and \(\ell\) is equivalent to \(\log_k(\ell)\) being irrational. Hence, Fact~\ref{thm:kronecker} tells us that for any non-empty open \(\mathcal{J} \subseteq (0,1)\), there is a \(t \in \mathbb{N}\) with \(\fr(t \log_k(\ell)) \in \mathcal{J}\). Density thus follows from the first part.
\end{proof}

\begin{corollary} \label{thm:always-solution}
  Let \(k\), \(\ell \in \mathbb{Z}_{\ge 2}\) be multiplicatively independent. Consider a system of strict linear homogeneous inequalities with two variables \(x\) and \(y\) and coefficients in \(\mathbb{Q}\). If this system has a solution \((u,v) \in \mathbb{R}_{>0 }^2\), then it also has infinitely many solutions \((u,v) \in k^{\mathbb{N}} \times \ell^{\mathbb{N}}\).
\end{corollary}
\begin{proof}
  Let \(S \subseteq \mathbb{R}_{> 0}^2\) be the set of solutions in \(\mathbb{R}_{>0}^2\). Since the inequalities are homogeneous, we have \((u,v) \in S\) if and only if \((\frac{u}{v}, 1) \in S\). Since the inequalities are strict and linear, the projection of \(S \cap \mathbb{R}_{>0} \times \{1\}\) to the first coordinate is an open interval. Since \(S \ne \emptyset \), this interval is also non-empty. Therefore, Lemma~\ref{thm:kronecker-dense} implies the existence of infinitely many \(m\), \(n \in \mathbb{N}\) with \((\frac{k^m}{\ell^n}, 1 ) \in S\).
\end{proof}

The next lemma can be seen as a generalization of Lemma \ref{thm:kronecker-dense} to more than two variables.
In what follows, we write tuples as \(x_{i : j}\) instead of \((x_i, x_{i+1}, \dots, x_j)\) to allow for more concise notation. When \(\ell \in \mathbb{Z}_{\ge 2}\) and \linebreak[2]\(t_{i:j} \in \mathbb{N}^{j-i+1}\), we also write \(\ell^{t}_{i:j}\) for \((\ell^{t_i}, \ell^{t_{i+1}}, \dots, \ell^{t_j})\).

\begin{lemma}[{Pumping Lemma, \cite[Lemma 13]{Karimov}}]
  \label{thm:pumping-lemma} 
  Let \(k\), \(\ell \in \mathbb{Z}_{\ge 2}\) be multiplicatively independent, \(u_{1:m} \in (k^{\mathbb{N}})^m\), \(v_{1:n} \in (\ell^{\mathbb{N}})^n\), and let \(h_1\), \dots, \(h_r\) be \(\mathbb{Q}\)-linear forms in \(m+n\) variables.
  Write \(J = \{j : h_j(u_{1:m}, v_{1:n}) > 0\}\). Then for any \(\varepsilon > 0\), there exists a non-empty open interval \(\mathcal{I} \subseteq (0,1)\) with the following property. For all \(t_1 \in \mathbb{N}\) with \(\fr(t_1\log_k(\ell)) \in \mathcal{I}\), there exist \(s_{1:m} \in \mathbb{N}^{m}\), \(t_{2:n} \in \mathbb{N}^{n-1}\) such that for all \(j\in \{1, \dots, r\}\),
  \begin{enumerate}
    \item if \(j\in J\), then \(h_j(k^s_{1:m}, \ell^t_{1:n}) > 0\), and
    \item \(\displaystyle \left|\frac{h_j(k^s_{1:m}, \ell^t_{1:n})}{\ell^{t_1}}-\frac{h_j(u_{1:m},v_{1:n})}{v_1}\right| < \varepsilon\).
  \end{enumerate}
\end{lemma}

\begin{definition} \label{def:archimedean-class}
  Let \((M,+,<)\) be an ordered group and \(a\), \(b\in M_{> 0}\). We write \(a \gg b\) if \(a > n \cdot b\) for all \(n\in \mathbb{N}\), and \(a \ll b\) if \(b \gg a\). We write \(a \sim b\) if neither \(a \gg b\) nor \(a \ll b\). The equivalence classes of the equivalence relation \(\sim\) on \(M_{>0}\) are called \emph{Archimedean classes}.
\end{definition}
The following is the essential consequence of the basic inequality axioms for non-standard powers.
\begin{lemma} \label{thm:rep-per-arch-class} 
  Let \(\ell \in  \mathbb{Z}_{\ge 2}\) and suppose \((\mathcal{M}, A_{\ell})\) satisfies \ref{ax:oag}, \ref{ax:universal-multiplication}, and \(\BInequ(\{\ell\})\). Let \(K\) be an Archimedean class of \(\mathcal{M}\). If \(K \cap A_{\ell} \ne \emptyset \), then \(K \cap A_{\ell} \subseteq \ell^{\mathbb{Z}}\cdot \alpha\) for any \(\alpha\in K \cap A_{\ell}\).
\end{lemma}
\begin{proof}
  Suppose \(\alpha\), \(\alpha' \in K \cap A_{\ell}\). Since \(\alpha\) and \(\alpha'\) lie in the same Archimedean class, there are \(m\), \(n\in \mathbb{N}\) such that \(\alpha < \ell^m \alpha' < \ell^n \alpha\). We pick \(n\) minimal with this property. Hence, \(\ell^{n-1} \alpha \le \ell^m \alpha'\). By \ref{ax:universal-multiplication}, we have \(\ell^m \alpha'\), \(\ell^{n-1} \alpha \in A_{\ell}\). Therefore, the basic inequality axioms tell us that there is no element of \(A_{\ell}\) strictly between \(\ell^{n-1}\alpha\) and \(\ell^{n}\alpha\). Hence, \(\ell^m \alpha' = \ell^{n-1} \alpha\) and thus \(\alpha' = \ell^{n-m-1} \alpha \in \ell^{\mathbb{Z}} \alpha\).
\end{proof}

The next proof follows the algorithm from \cite[Section 6]{Karimov} for deciding solvability of inequalities in powers.

\begin{lemma} \label{thm:algorithm} 
  Let \(k\), \(\ell \in \mathbb{Z}_{\ge 2}\) be multiplicatively independent, let \((\mathcal{M},A_k, A_{\ell})\) satisfy the axioms \ref{ax:oag}, \ref{ax:universal-multiplication}, and \(\BInequ(\{k, \ell\})\), and let \(C \in \mathbb{Q}^{r \times (m+n)}\). If there is \(\alpha \in (A_k)^m \times (A_{\ell})^n\) with \(C\alpha > 0\), then there are infinitely many \(z \in (k^{\mathbb{N}})^m \times (\ell^{\mathbb{N}})^n\) with \(Cz > 0\).
\end{lemma}
\begin{proof}
  The proof proceeds by induction on the number of variables. The base case \(m \le 1\) and \(n\le 1\) is contained in Lemma \ref{thm:always-solution}. Let \(\alpha = (\alpha_{1:m},\alpha'_{1:n})\) with \(\alpha_{1:m} \in (A_k)^m\), \(\alpha_{1:n}' \in (A_{\ell})^{n}\) satisfy \(C \alpha > 0\). By Lemma~\ref{thm:rep-per-arch-class}, if \(\alpha_i\) and \(\alpha_j\) are in the same Archimedean class, then \(\alpha_j = k^\nu \alpha_i\) for some \(\nu\in \mathbb{Z}\), so we may reduce the number of variables and apply induction. Therefore, we can now assume \(\alpha_1 \gg \alpha_2 \gg \dots \gg \alpha_m\), and similarly \(\alpha_1' \gg \alpha_2' \gg \dots \gg \alpha_n'\). 

  Without loss of generality, assume \(\alpha_1 > \alpha_1'\). We consider the system of inequalities \(C(x_{1:m},y_{1:n}) > 0\). By dividing by the coefficient of \(x_1\) whenever it is non-zero and bringing all other terms to the other side, we can re-write it as
  \begin{align*}
    b_i^-y_1 + h_i^-(x_{2:m}, y_{2:n}) &< x_1 < b_j^+y_1 + h_j^+(x_{2:m}, y_{2:n}) \qquad \text{for } i\in I, j\in J,\\
    \text{and }&\psi(x_{2:m}, y_{1:n})
  \end{align*}
  where \(I\), \(J\) are finite index sets, \(b_i^-\), \(b_j^+ \in \mathbb{Q}\), \(h_i^-\), \(h_j^+\) are \(\mathbb{Q}\)-linear forms for \(i \in I\), \(j \in J\), and \(\psi\) is a system of strict homogeneous \(\mathbb{Q}\)-linear inequalities. This \(\psi\) collects all the inequalities in which the coefficient of \(x_1\) was zero. If \(J = \emptyset \), then we can find \((u_{2:m},v_{1:n}) \in (k^{\mathbb{N}})^{m-1} \times (\ell^{\mathbb{N}})^n\) satisfying \(\psi(u_{2:m},v_{1:n})\) by induction and pick \(u_1 \in k^{\mathbb{N}}\) large enough such that \(b_i^-v_1+h_i^-(u_{2:m},v_{2:n}) < u_1\) for all \(i \in I\). So let \(J \ne \emptyset \) from now on. By assumption, we have
  \begin{equation}
  \label{eq:1}
    b_i^-\alpha'_1 + h_i^-(\alpha_{2:m}, \alpha'_{2:n}) < \alpha_1 < b_j^+\alpha'_1 + h_j^+(\alpha_{2:m}, \alpha'_{2:n}).
  \end{equation}
  Since \(J \ne \emptyset \), we cannot have \(\alpha_1 \gg \alpha_1'\). Hence, \(\alpha_1 > \alpha_1'\) implies that \(\alpha_1 \sim \alpha_1'\). Now \(\alpha_1' \gg \alpha_2, \alpha_2'\) yields that 
  \begin{align*}
    |b_i^- \alpha'_1| &\gg \left|h_i^-(\alpha_{2:m}, \alpha'_{2:n})\right|\qquad\text{and}\\
    |b_j^+ \alpha'_1| &\gg \left|h_j^+(\alpha_{2:m}, \alpha'_{2:n})\right|
  \end{align*}
  for any \(i\in I\), \(j\in J\) with \(b_i^-\), \(b_j^+ \ne 0\).  Therefore, we must have \(b_i^- \le b_j^+\) and \(b_j^{+} >0 \) for any \(i\in I\), \(j\in J\) for \eqref{eq:1} to hold. Applying the same reasoning to \(\psi\) shows that the coefficient of \(y_1\) in every inequality in \(\psi\) needs to be non-negative. We may thus split \(\psi(x_{2:m},y_{1:n})\) into \(\psi^+(x_{2:m},y_{1:n}) \wedge \psi^0(x_{2:m},y_{2:n})\) such that the coefficient of \(y_1\) in every inequality of \(\psi^+\) is positive.

  To now solve these inequalities, we first consider the easier case that \(b_i^- < b_j^+\) for all \(i\in I\), \(j\in J\). By induction, there are infinitely many solutions \((u_{2:m}, v_{2:n}) \in (k^{\mathbb{N}})^{m-1} \times (\ell^{\mathbb{N}})^{n-1}\) of \(\psi^0\). Fix one particular such \((u_{2:m}, v_{2:n})\). By Lemma \ref{thm:kronecker-dense}, we can find infinitely many \(u_1 \in k^{\mathbb{N}}\) and \(v_1 \in \ell^{\mathbb{N}}\) such that \(b_i^- < u_1/v_1 < b_j^+\) for all \(i\in I\), \(j\in J\). This means that for large enough \(v_1\), we will have
  \[b_i^-+ \frac{h_i^-(u_{2:m}, v_{2:n})}{v_1} < \frac{u_1}{v_1}  < b_j^+ + \frac{h_j^+(u_{2:m}, v_{2:n})}{v_1} \qquad \text{for } i\in I, j\in J.\]
  Furthermore, \(\psi^+(u_{2:m},v_{1:n})\) will also hold for sufficiently large \(v_1\), which means we have found the desired solutions \((u_{1:m},v_{1:n})\) in this case.

  Now consider the case that \(b := b_i^- = b_j^+\) for some \(i\in I\), \(j\in J\). Set \(I' := \{i \in I : b_i^- = b\}\) and \(J' := \{j \in J : b_j^+ = b\}\). Then
  \[b\alpha'_1 + h_i^-(\alpha_{2:m}, \alpha'_{2:n}) < \alpha_1 < b\alpha'_1 + h_j^+(\alpha_{2:m}, \alpha'_{2:n})\]
  implies that
  \[h_i^-(\alpha_{2:m}, \alpha'_{2:n}) < h_j^+(\alpha_{2:m}, \alpha'_{2:n}) \qquad \text{for all } i \in I', j \in J'.\]
  So by induction, there are infinitely many \((u_{2:m},v_{2:n}) \in (k^{\mathbb{N}})^{m-1} \times (\ell^{\mathbb{N}})^{n-1}\) satisfying both 
  \[h_i^-(u_{2:m},v_{2:n}) < h_j^+(u_{2:m},v_{2:n}) \qquad \text{for all } i \in I', j \in J'\]
  and \(\psi^0(u_{2:m},v_{2:n})\). Fix one particular such \((u_{2:m},v_{2:n})\), and pick \(\varepsilon > 0\) such that
  \[\frac{h_i^-(u_{2:m},v_{2:n})}{v_2} + \varepsilon < \frac{h_j^+(u_{2:m},v_{2:n})}{v_2} - \varepsilon \qquad \text{for all } i \in I', j \in J'.\]
  By Lemma \ref{thm:pumping-lemma}, there is a non-empty open interval \(\mathcal{I} \subseteq (0,1)\) such that for any \(t_2 \in \mathbb{N}\) with \(\fr(t_2\log_k(\ell)) \in \mathcal{I}\), there exist \(s_{2:m} \in \mathbb{N}^{m-1}\) and \(t_{3:n} \in \mathbb{N}^{n-2}\) such that
  \begin{equation}
  \label{eq:2}
    \frac{h_i^-(k^s_{2:m}, \ell^t_{2:n})}{\ell^{t_2}} < \frac{h_i^-(u_{2:m},v_{2:n})}{v_2} + \varepsilon < \frac{h_j^+(u_{2:m},v_{2:n})}{v_2} - \varepsilon < \frac{h_j^+(k^s_{2:m}, \ell^t_{2:n})}{\ell^{t_2}}
  \end{equation}
  for all \(i \in I'\), \(j \in J'\), and \(\psi^0(k^s_{2:m}, \ell^t_{2:n})\) holds. On the other hand, we can use Lemma \ref{thm:kronecker-dense} to find a family \((\mathcal{J}_w)_{w \in \mathbb{N}}\) of non-empty open subintervals of \((0,1)\) such that for any \(t_1\), \(w \in \mathbb{N}\) with \(\fr(t_1 \log_k(\ell)) \in \mathcal{J}_w\), there exists \(s_1 \in \mathbb{N}\) such that
  \begin{equation}
  \label{eq:3}
    b + \frac{1}{\ell^w} \left(\frac{h_i^-(u_{2:m},v_{2:n})}{v_2} + \varepsilon \right) < \frac{k^{s_1}}{\ell^{t_1}} < b + \frac{1}{\ell^w} \left(\frac{h_j^+(u_{2:m},v_{2:n})}{y_2} - \varepsilon \right)
  \end{equation}
  for all \(i \in I'\), \(j \in J'\). Our goal is now to find \(w \in \mathbb{N}\) and \(t_1\), \(t_2 \in \mathbb{N}\) such that
  \begin{equation}
  \label{eq:4}
    \fr(t_1 \log_k(\ell)) \in \mathcal{J}_w, \quad \fr(t_2\log_k(\ell)) \in \mathcal{I} \quad \text{ and } \quad w = t_1-t_2,
  \end{equation}
  as then \eqref{eq:2} and \eqref{eq:3} together yield
  \begin{align*}
    \frac{k^{s_1}}{\ell^{t_1}} &< b + \frac{1}{\ell^{t_1-t_2}} \left(\frac{h_j^+(u_{2:m},v_{2:n})}{v_2} - \varepsilon \right) \\
    &< b + \frac{1}{\ell^{t_1-t_2}} \frac{h_j^+(k^s_{2:m},\ell^t_{2:n})}{\ell^{t_2}} \\
    &= b + \frac{h_j^+(k^s_{2:m},\ell^t_{2:n})}{\ell^{t_1}} \qquad\qquad \text{for all } j \in J'
  \end{align*}
  and similarly for the lower bound. Notice that \eqref{eq:4} is equivalent to
  \[\fr(t_1 \log_k(\ell)) \in \mathcal{J}_w \cap \fr(\mathcal{I}+w\log_k(\ell)) \quad \text{ and } \quad t_2 = t_1-w. \]
  Thus, we need to find \(w \in \mathbb{N}\) such that
  \[\mathcal{J}_w \cap \fr(\mathcal{I}+w\log_k(\ell)) \ne \emptyset.\]
  We claim that there are infinitely many such \(w\) (see Figure \ref{fig:circle} for an illustration of this part of the proof). Indeed, observe that if \(s_1\), \(t_1 \in \mathbb{N}\) satisfy \eqref{eq:3}, then as \(w \to \infty \), we have \(k^{s_1}/\ell^{t_1} \to b\) and thus \[\fr(t_1\log_k(\ell)) \to \fr(-\log_k(b)).\] So the intervals \(\mathcal{J}_w\) shrink and converge to the point \(\fr(-\log_k(b))\) as \(w \to \infty \). Let \(\mathcal{J} \subseteq (0,1)\) be an interval centered at \(\fr(-\log_k(b))\) whose length is at most half the length of \(\mathcal{I}\), and pick \(w_0 \in \mathbb{N}\) such that for all \(w \ge w_0\), we have \(\mathcal{J}_w \subseteq \mathcal{J}\). Let \(\xi\) be the midpoint of \(\mathcal{I}\). By Fact \ref{thm:kronecker}, there are infinitely many \(w \in \mathbb{N}\) such that \(\fr(\xi+w\log_k(\ell)) \in \mathcal{J}\) and thus \(\mathcal{J} \subseteq \fr(\mathcal{I}+w\log_k(\ell))\). Hence, for infinitely many \(w \ge w_0\), the sets \(\fr(\mathcal{I}+w\log_k(\ell))\) and \(\mathcal{J}_w\) have non-empty intersection.

  So we have shown that there are infinitely many \(u_{1:m} \in (k^{\mathbb{N}})^m\) and \(v_{1:n} \in (\ell^{\mathbb{N}})^n\) satisfying the inequalities for \((i,j) \in I' \times J'\) and \(\psi^0\). As explained before, the other inequalities will be satisfied as well for sufficiently large \(v_1\), which finishes the proof.
\end{proof}

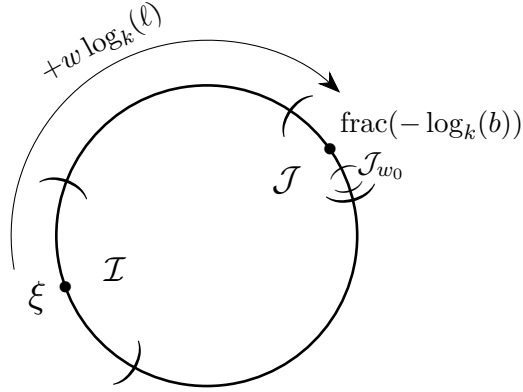
\begin{figure}[ht]
  \centering
  \begin{tikzpicture}[scale = 2]
    \draw[line width=1pt] (0,0) circle (1);

    \node at (35:1) {\(\bullet\)};
    \node[anchor=south west] at (35:1) {\(\fr(-\log_k(b))\)};
    \node[rotate=-35,scale=1.75] at (55:1) {$($};
    \node[rotate=-75,scale=1.75] at (15:1) {$)$};
    \node[scale = 1.2] at (35:.65) {$\mathcal{J}$};

    \node[rotate=-64,scale=1] at (26:1) {$($};
    \node[rotate=-72,scale=1] at (18:1) {$)$};
    \node[] at (22:1.25) {$\mathcal{J}_{w_0}$};

    \node at (200:1) {\(\bullet\)};
    \node[scale=1.3] at (200:1.2) {\(\xi\)};
    \node[rotate=150,scale=1.75] at (240:1) {$($};
    \node[rotate=70,scale=1.75] at (160:1) {$)$};
    \node[scale=1.2] at (200:.65) {$\mathcal{I}$};

    \draw[-{Stealth[scale=1.8]}] (190:1.3) arc (190:47:1.3) node[midway, above, sloped] {\(+ w \log_k(\ell)\)};
  \end{tikzpicture}
  \caption{Illustration of the last part of the proof of Lemma \ref{thm:algorithm}}
  \label{fig:circle} 
\end{figure}

\begin{proposition} \label{thm:inequ-and-congruence} 
  Let \((\mathcal{M}, A_k, A_{\ell}) \models T_{\forall }(\{k, \ell\})\), and let \(C \in \mathbb{Q}^{r \times (m+n)}\). Let \(\alpha_{1:m} \in (A_k \setminus \mathbb{Z})^m\) and \(\alpha_{1:n}' \in (A_{\ell} \setminus \mathbb{Z})^n\) be such that \(C(\alpha_{1:m},\alpha'_{1:n}) >0 \), and let \(\Delta(x_{1:m},y_{1:n}) \subseteq d_{\mathcal{M}}(\alpha_{1:m},\alpha'_{1:n})\) be finite. Then there exist \(u_{1:m} \in (k^{\mathbb{N}})^m\) and \(v_{1:n} \in (\ell^{\mathbb{N}})^n\) satisfying \(C(x_{1:m},y_{1:n}) > 0\) and \(\Delta(x_{1:m},y_{1:n})\).
\end{proposition}
\begin{proof}
  By the Carmichael axioms (Lemma \ref{thm:carmichael-use}), we can assume
  \[\Delta = \left\{ D_d(x_i - k^{s_i}), D_d(y_j - \ell^{t_j}) : d \in \mathcal{D}, i \in \{1, \dots, m\}, j \in \{1, \dots, n\}\right\} \]
  for some \(s_{1:m} \in \mathbb{N}^m\), \(t_{1:n} \in \mathbb{N}^n\), and some finite set of prime powers \(\mathcal{D}\). Pick \(\rho \in \mathbb{N}_{>0}\) that is divisible by \(\lambda(d)\) for all \(d \in \mathcal{D}\) and larger than \(s_i\) and \(t_j\) for all \(i \in \{1, \dots, m\}\), \(j \in \{1, \dots, n\}\). Now if \(d \in \mathcal{D}\) is coprime to \(k\), then \(k^{\rho} \equiv 1 \pmod{d}\), so \(D_d(x_i-k^{s_i})\) is equivalent to \(D_d(k^{\rho-s_i}x_i - 1)\). On the other hand, if \(d \in \mathcal{D}\) is not coprime to \(k\), then since \(d\) is a prime power, there is some \(e \in \mathbb{N}\) such that \(d \mid k^e\). Since \(\alpha_i \notin \mathbb{Z}\), the Carmichael axioms (Definition~\ref{def:carmichael-axioms}~\eqref{eq:car1}) tell us that \(\mathcal{M} \models D_{k^{e}}(\alpha_i)\), so in particular \(\mathcal{M} \models D_d(\alpha_i)\). Hence, \(k^{s_i}\) must be divisible by \(d\). Therefore, \(D_d(x_i-k^{s_i})\) is equivalent to \(D_d(x_i)\), which is satisfied for all sufficiently large \(x_i \in k^{\mathbb{N}}\). Using the same reasoning for \(\ell\), we obtain that if \(\tilde{u}_{1:m} \in (k^{\mathbb{N}})^m\) and \(\tilde{v}_{1:n} \in (\ell^{\mathbb{N}})^n\) are sufficiently large and satisfy
  \begin{align*}
    \widetilde{\Delta}(\tilde{x}_{1:m},\tilde{y}_{1:n}) := \quad &\left\{ D_d(\tilde{x}_i-1) : d \in \mathcal{D} \text{ coprime to } k, i \in \{1, \dots, m\}\right\} \\
    \cup &\left\{ D_d(\tilde{y}_j-1) : d \in \mathcal{D} \text{ coprime to } \ell, j \in \{1, \dots, n\}\right\},
  \end{align*}
  then \(u_{1:m} \in (k^{\mathbb{N}})^m\) and \(v_{1:n} \in (\ell^{\mathbb{N}})^n\) obtained through \(u_i = \tilde{u}_i/k^{\rho-s_i}\) and \(v_j = \tilde{v}_j / \ell^{\rho-t_j}\) satisfy \(\Delta(u_{1:m},v_{1:n})\).

  Let \(\widetilde{C} \in \mathbb{Q}^{r \times (m+n)}\) be the matrix obtained from \(C\) by multiplying the elements in the \(i\)-th column by \(k^{\rho-s_i}\) for \(i \in \{1, \dots, m\}\), and by multiplying the elements in the \((m+j)\)-th column by \(\ell^{\rho-t_j}\) for \(j \in \{1, \dots, n\}\). It suffices to find sufficiently large \(\tilde{u}_{1:m} \in (k^{\mathbb{N}})^m\) and \(\tilde{v}_{1:n} \in (\ell^{\mathbb{N}})^n\) satisfying \(\widetilde{C}(\tilde{u}_{1:m},\tilde{v}_{1:n}) > 0\) and \(\widetilde{\Delta}(\tilde{u}_{1:m},\tilde{v}_{1:n})\). To do this, we set \(\tilde{k} := k^{\rho}\) and \(\tilde{\ell} := \ell^{\rho}\). Since all elements of \((\tilde{k}^{\mathbb{N}})^m \times (\tilde{\ell}^{\mathbb{N}})^n\) satisfy \(\widetilde{\Delta}\), it suffices to find sufficiently large \((\tilde{x}_{1:m}, \tilde{y}_{1:n}) \in (\tilde{k}^{\mathbb{N}})^m \times (\tilde{\ell}^{\mathbb{N}})^n\) with \(\widetilde{C}(\tilde{x}_{1:m},\tilde{y}_{1:n}) > 0\). Define 
  \[A_{\tilde{k}} := \tilde{k}^{\mathbb{N}} \cup \bigcup_{i=1}^m (\tilde{k}^{\mathbb{Z}}\cdot k^{\rho-s_i}\alpha_i) \cap M \quad\text{ and }\quad A_{\tilde{\ell}} := \tilde{\ell}^{\mathbb{N}} \cup \bigcup_{j=1}^n (\tilde{\ell}^{\mathbb{Z}}\cdot \ell^{\rho-t_j}\alpha'_j) \cap M.\]
  Now \((\mathcal{M},A_{\tilde{k}},A_{\tilde{\ell}})\) satisfies \ref{ax:oag}, \ref{ax:universal-multiplication} (for \(L = \{\tilde{k}, \tilde{\ell}\}\)), and \(\BInequ(\{\tilde{k}, \tilde{\ell}\})\). Applying Lemma~\ref{thm:algorithm} finishes the proof.
\end{proof}

\subsection{Constructing an Embedding}

\begin{lemma} \label{thm:unique-rep}
  Let \((\mathcal{M},(A_{\ell})_{\ell\in L}) \models T_{\forall }(L)\), and let \(\alpha_1\), \dots, \(\alpha_n \in A \setminus \mathbb{Z}\) be such that for every \(\ell\in L\) and every Archimedean class \(K\) of \(\mathcal{M}\), there is at most one \(i\in \{1, \dots, n\}\) such that \(\alpha_i \in A_{\ell} \cap K\). Then \(\sum_{i=1}^n c_i\alpha_i \notin \mathbb{Z}\) for all \(c_1\), \dots, \(c_n \in \mathbb{Z} \setminus \{0\}\).
\end{lemma}
\begin{proof}
  The claim is trivially true for \(n=1\), so let \(n \ge 2\) from now on. Let \(c_1\), \dots, \(c_n \in \mathbb{Z} \setminus \{0\}\). Assume, for the sake of contradiction, that \(\sum_{i=1}^n c_i\alpha_i = b \in \mathbb{Z}\). Now \(b \ne 0\) would immediately contradict the Mann axioms, so we have \(b= 0\) instead. This means we can apply the Mann axioms to \(\sum_{i=1}^{n-1} c_i\alpha_i = c_n\alpha_n\). This yields that \(\alpha_1\), \dots, \(\alpha_n \in A_{\ell}\) for some \(\ell \in L\), and that for each \(i\in\{1, \dots, n-1\}\) there are \(s\), \(t\in \ell^{\mathbb{N}}\) such that \(s\alpha_i = t\alpha_n\). In particular, we have \(s\alpha_1 = t\alpha_n\) for some \(s\), \(t\in \mathbb{Z}_{>0}\), which contradicts our assumption that \(\alpha_1\) and \(\alpha_n\) are not in the same Archimedean class.
\end{proof}

\begin{lemma} \label{thm:embed-powers}
  Let \((\mathcal{M}, (A_{\ell})_{\ell\in L}) \models T_{\forall }(L)\). Furthermore, let \(\kappa > |M|\) be a cardinal, and let \((\mathcal{N}, (B_{\ell})_{\ell\in L}) \models \Th(\mathcal{Z})\) be \(\kappa\)-saturated. Then there exists an \(\mathcal{L}_{<}\)-embedding \(f \colon \left\langle A \right\rangle\div \hookrightarrow \mathcal{N}\).
\end{lemma}
\begin{proof}
  For every \(\ell \in L\), let \((\alpha_i^{\ell})_{i\in I_\ell}\) be a family containing exactly one element of \(K \cap A_{\ell}\) for every Archimedean class \(K \ne \mathbb{Z}\) of \(\mathcal{M}\) with \(K \cap A_{\ell} \ne \emptyset \). By Lemma \ref{thm:rep-per-arch-class}, if \(K \ne \mathbb{Z}\) is an Archimedean class of \(\mathcal{M}\) and \(\alpha_i^{\ell}\in K\), then
  \[K \cap A_{\ell} \subseteq \ell^{\mathbb{Z}}\cdot \alpha_i^{\ell},\]
  which implies that
  \[\left\langle \{1\} \cup \{\alpha_i^{\ell} : \ell\in L, i\in I_{\ell}\} \right\rangle\div = \left\langle A \right\rangle\div.\]

  Let \(\overline{\alpha}\) be the tuple \((\alpha_i^{\ell} : \ell\in L, i\in I_{\ell})\), and let \(\overline{x}\) be the (possibly infinite) tuple of variables \((x_i^{\ell} : \ell\in L, i\in I_{\ell})\). Define \(p(\overline{x})\) as the type consisting of the formulas 
  \[\sum_{\ell\in L} \sum_{i\in I_\ell} c_i^\ell x_i^\ell >0,\]
  for all families of integers \((c_i^{\ell} : \ell \in L, i \in I_{\ell})\) such that \(c_i^{\ell} \ne 0\) for only finitely many pairs \((i,\ell)\), and 
  \[ \mathcal{M}\models \sum_{\ell\in L} \sum_{i\in I_{\ell}} c_i^{\ell}\alpha_i^{\ell} > 0.\]
  Furthermore, let
  \[q(\overline{x}) := p(\overline{x}) \cup \left\{U_\ell(x_i^{\ell}) : \ell\in L, i\in I_{\ell}\right\} \cup d_{\mathcal{M}}(\overline{\alpha}).\]
  We claim that every finite subset of \(q(\overline{x})\) is realized in \(\mathcal{Z}\). Indeed, this follows directly from the inequality axioms for \(|L| \ge 3\) and from Proposition \ref{thm:inequ-and-congruence} for \(|L| \le 2\). By saturation, there is a realization \(\overline{\beta} = (\beta_i^{\ell} : \ell\in L, i\in I_{\ell})\) of \(q(\overline{x})\) in \(\mathcal{N}\). Using this, we obtain our embedding \(f \colon \left\langle A \right\rangle\div \to \mathcal{N}\) by defining
  \begin{align*}
    f \left( \frac{g + \sum_{i, \ell} c_i^{\ell} \alpha_i^{\ell}}{n}\right) &:= \frac{g + \sum_{i, \ell} c_i^{\ell} \beta_i^{\ell}}{n}
  \end{align*}
  for \(g\in \mathbb{Z}\), \(n\in \mathbb{Z}_{>0}\) and \(c_i^{\ell} \in \mathbb{Z}\) with \(c_i^{\ell} \ne 0\) for only finitely many \(\ell\in L\), \(i\in I_{\ell}\).
  To make sure that this map is well-defined, we need to verify two things: First, we have \(d_{\mathcal{M}}(\overline{\alpha}) \subseteq d_{\mathcal{N}}(\overline{\beta})\), which means that if \(g + \sum c_i^{\ell} \alpha_i^{\ell}\) is divisible by \(n\), then so is \(g + \sum c_i^{\ell} \beta_i^{\ell}\). Secondly, every element \(a\in \left\langle A \right\rangle\div\) has a unique representation of the form \(a = \frac{1}{n}(g + \sum c_i^{\ell} \alpha_i^{\ell})\) up to scaling by an integer. Indeed, if
  \[\frac{g+ \sum_{i, \ell} c_i^{\ell} \alpha_i^{\ell}}{n} = \frac{h+\sum_{i, \ell} d_i^{\ell} \alpha_i^{\ell}}{m},\]
  then applying Lemma \ref{thm:unique-rep} to
  \[\sum_{i, \ell} \left(mc_i^{\ell}-nd_i^{\ell}\right)\alpha_i^{\ell} = nh-mg\in \mathbb{Z}\]
  tells us that \(mc_i^{\ell} = nd_i^{\ell}\) for all \(\ell\in L\), \(i\in I_{\ell}\). Thus, also \(nh = mg\), and hence
  \[\frac{g+ \sum_{i, \ell} c_i^{\ell} \beta_i^{\ell}}{n} = \frac{h+\sum_{i, \ell} d_i^{\ell} \beta_i^{\ell}}{m}.\]

  It is clear that \(f\) is a group homomorphism. We now check that \(f\) is strictly increasing. Suppose \(a = g + \sum_{i, \ell} c_i^{\ell} \alpha_i^{\ell} > 0\). If \(c_i^{\ell} = 0\) for all \(\ell\in L\), \(i\in I_{\ell}\), then \(f(a) = a >0\). Otherwise, we have \(\sum c_i^{\ell} \alpha_i^{\ell} \notin \mathbb{Z}\) by the Mann axioms (Lemma \ref{thm:unique-rep}). Now the discreteness axiom \ref{ax:discrete} implies that \(\sum c_i^{\ell} \alpha_i^{\ell} > \mathbb{Z}\) or \(\sum c_i^{\ell} \alpha_i^{\ell} < \mathbb{Z}\). The latter cannot be true, as it would imply \(a <0\). So in particular we have \(\sum c_i^{\ell} \alpha_i^{\ell} > 0\) and thus \(\sum c_i^{\ell} \beta_i^{\ell} > 0\), since \(\overline{\beta}\) satisfies \(p(\overline{x})\). Because \(\mathcal{N}\) also satisfies the Mann axioms and the discreteness axiom, we get \(\sum c_i^{\ell} \beta_i^{\ell} > \mathbb{Z} \) and thus \(f(a) = g + \sum c_i^{\ell} \beta_i^{\ell} > 0\), as desired.

  Finally, we verify that \(f\) is an \(\mathcal{L}\)-homomorphism. This proceeds similarly as in Proposition \ref{thm:back-and-forth-system}: Suppose
  \[a = \frac{g + \sum_{i, \ell} c_i^{\ell} \alpha_i^{\ell}}{n} \in A_k\]
  for some \(k \in L\). If \(\sum c_i^{\ell} \alpha_i^{\ell} = 0\), then \(f(a)=a \in B_k\). Now suppose \(\sum c_i^{\ell} \alpha_i^{\ell} \ne 0\) and thus \(a \notin \mathbb{Z}\) by the Mann axioms. Rearranging yields
  \[na - \sum_{i, \ell} c_i^{\ell} \alpha_i^{\ell} = g\in \mathbb{Z}\]
  and thus \(g=0\) by again invoking the Mann axioms. Applying the Mann axioms to \(\sum c_i^{\ell} \alpha_i^{\ell} = na\) shows that there exist \(j \in I_k\) and \(s\), \(t\in k^{\mathbb{N}}\) such that \(c_j^k \ne 0\) and \(s \cdot \alpha_j^k = t \cdot a\). Hence, \(f(a) = \frac{s_i}{t} \beta_i^k \in B_k\) by \ref{ax:universal-multiplication}, as desired. The other implication is proven analogously.
\end{proof}

\begin{lemma} \label{thm:avoiding-the-powers}
  Let \((\mathcal{M},(A_{\ell})_{\ell\in L}) \models T_{\forall }(L)\), let \(\kappa > |M|\) be a cardinal, and let \((\mathcal{N}, (B_{\ell})_{\ell \in L}) \models \Th(\mathcal{Z})\) be \(\kappa\)-saturated. Let \(M'\) be a divisibly closed subgroup of \(M\) containing \(1\), let \(f \colon M' \to \mathcal{N}\) be an \(\mathcal{L}_{<}\)-embedding, and let \(\alpha \in M \setminus M'\). Then there exists \(\beta\in N \setminus \left\langle f(M'), B \right\rangle\div \) such that \(d_{\mathcal{M}}(\alpha) \subseteq d_{\mathcal{N}}(\beta)\) and
  \begin{equation}
  \label{eq:same-cut}
  a < c\alpha < b \quad \text{implies} \quad f(a) < c\beta < f(b) \quad \text{for all } a,b\in M', c\in \mathbb{Z}_{>0}.
  \end{equation}
\end{lemma}
\begin{proof}
  Let \(p(x)\) be the type consisting of the formulas
  \[f(a) < c \cdot x < f(b)\]
  for all \(a\), \(b \in M'\), \(c \in \mathbb{Z}_{>0}\) such that \(\mathcal{M} \models a < c \alpha < b\). Furthermore, let \(q(x)\) be the type consisting of the formulas
  \[\neg U_\ell \left(\frac{f(g)+k \cdot x}{n} \right)\]
  for all \(\ell \in L\), \(g \in M'\), and \(k\), \(n \in \mathbb{Z} \setminus \{0\}\) such that \(\mathcal{M} \models D_n(g+k\alpha)\).
  Our desired \(\beta\) would be a realization of the type \(d_{\mathcal{M}}(\alpha) \cup p(x) \cup q(x)\).
  By saturation of \(\mathcal{N}\), it suffices to show that this type is finitely satisfiable in \(\mathcal{N}\). Let \(\Delta \subseteq d_{\mathcal{M}}(\alpha)\) be finite, \(a_i\), \(b_i \in M'\) and \(c_i\in \mathbb{Z}_{>0}\) with
  \[\mathcal{M} \models a_i < c_i\alpha < b_i\]
  for \(i \in \{1, \dots, m\}\), and let \(g_i\in M'\), \(k_i\), \(n_i \in \mathbb{Z} \setminus \{0\}\) be such that
  \[\mathcal{M} \models D_{n_i}(g_i+k_i\alpha)\]
  for \(i\in \{1, \dots, r\}\).
  Set \(c := c_1 \cdots c_m\) and \(a := \max_i \frac{c}{c_i} a_i\), \(b := \min_i \frac{c}{c_i} b_i\). Then \(f(a_i) < c_i x < f(b_i)\) for all \(i\in \{1, \dots, m\}\) is equivalent to \(f(a) < cx < f(b)\). Since \(M'\) is divisibly closed and contains \(1\), we have \(c\alpha + k \notin M'\) for every \(k\in \mathbb{Z}\). This implies \(b-a > \mathbb{Z}\), which means \(f(b)-f(a) > \mathbb{Z}\) in \(\mathcal{N}\). So there exists \(\beta'\in N\) such that \(f(a) < c(\beta' + \mathbb{Z}) < f(b)\). Since \(\mathcal{N}\) is a model of Presburger arithmetic, \(\beta'+\mathbb{Z}\) contains infinitely many realizations of \(\Delta\). Furthermore, for each \(i\in \{1, \dots, r\}\), there is at most one \(\beta\in \beta' + \mathbb{Z}\) such that
  \[\frac{f(g_i)+k_i \beta}{n_i}\in B,\]
  as the existence of two distinct infinitely large elements of \(B\) with finite difference would contradict the Mann axioms. This means we can pick \(\beta \in c\beta' + \mathbb{Z}\) in such a way that \(\beta\) both realizes \(\Delta\) and satisfies \(\frac{1}{n_i} (f(g_i) + k_i \beta) \notin B\) for every \(i\in \{1, \dots, r\}\), which concludes the proof.
\end{proof}

\begin{theorem} \label{thm:universal}
  The theory \(T_{\forall }(L)\) axiomatizes the universal theory of \(\mathcal{Z}\).
\end{theorem}
\begin{proof}
  Let \((\mathcal{M}, (A_\ell)_{\ell\in L}) \models T_{\forall }(L)\). It suffices to show that \(\mathcal{M}\) embeds in a model of \(\Th(\mathcal{Z})\). For this, we pick a cardinal \(\kappa > |M|\) and a \(\kappa\)-saturated \((\mathcal{N}, (B_\ell)_{\ell\in L}) \models \Th(\mathcal{Z})\). In Lemma~\ref{thm:embed-powers}, we have constructed an embedding \(f \colon \left\langle A \right\rangle \div \hookrightarrow \mathcal{N}\). We now inductively extend this to an embedding of all \linebreak[2]of \(\mathcal{M}\).

  Suppose we have an \(\mathcal{L}_{<}\)-embedding \(f \colon M' \hookrightarrow \mathcal{N}\) for a divisibly closed subgroup \(M' < M\) with \(\left\langle A \right\rangle\div \subseteq M'\), and suppose that \(\alpha \in M \setminus M'\). Let \(\beta \in N\) be as in Lemma~\ref{thm:avoiding-the-powers}. With that, we extend the embedding by
  \begin{align*}
    F\colon \left\langle M', \alpha \right\rangle\div  &\to \mathcal{N}\\
    \frac{g + k\alpha}{n}  &\mapsto \frac{f(g)+k\beta}{n}.
  \end{align*}
  This map is well-defined, since \(d_{\mathcal{M}}(\alpha) \subseteq d_{\mathcal{N}}(\beta)\). It is straightforward to verify that \(F\) is a group homomorphism. Property \eqref{eq:same-cut} ensures that \(F\) is strictly increasing. It remains to show that \(F\) is an \(\mathcal{L}\)-homomorphism. Let \(a\in \left\langle M', \alpha \right\rangle\div\) and \(\ell\in L\). If \(a \in A_\ell\), then \(a \in M'\), since \(A \subseteq M'\), so \(F(a) = f(a) \in B_{\ell}\). Furthermore, if \(F(a) \in B_{\ell}\), then \(F(a) \in f(M')\), since \(\beta \notin \left\langle f(M'), B \right\rangle\div\). Hence, by the injectivity of \(F\), we have \(a \in M'\) and thus \(F(a) = f(a) \in B_{\ell}\), which implies \(a \in A_{\ell}\). So \(F\) is indeed an \(\mathcal{L}_{<}\)-embedding.
\end{proof}

\bibliographystyle{amsplain}
\bibliography{references}


\end{document}